\documentclass[12pt, a4paper]{amsart}

\usepackage{amsmath, amsfonts, amssymb, amsthm}
\usepackage[margin=1.25in, top=1.5in, bottom=1in]{geometry}
\usepackage{tikz}

\newcommand{\R}{\mathbb{R}}
\newcommand{\C}{\mathbb{C}}

\newcommand{\N}{\mathbb{N}}
\newcommand{\T}{\mathbb{T}}
\newcommand{\Z}{\mathbb{Z}}

\newcommand{\Aut}{\operatorname{Aut}}

\newcommand{\TT}{\mathcal{T}}
\newcommand{\OO}{\mathcal{O}}
\newcommand{\FF}{\mathcal{F}}
\newcommand{\KK}{\mathcal{K}}
\newcommand{\LL}{\mathcal{L}}

\newcommand{\id}{\operatorname{id}}

\newcommand{\supp}{\operatorname{supp}}
\newcommand{\lsp}{\operatorname{span}}
\newcommand{\clsp}{\overline{\operatorname{span}}}

\def\dashind{\operatorname{\!-Ind}}
\allowdisplaybreaks
\newtheorem{thm}{Theorem}
\newtheorem{lemma}[thm]{Lemma}

\newtheorem{cor}[thm]{Corollary}
\newtheorem{prop}[thm]{Proposition}

\theoremstyle{definition}

\newtheorem{example}[thm]{Example}
\newtheorem{remark}[thm]{Remark}

\numberwithin{equation}{section}
\numberwithin{thm}{section}

\title[KMS states]{\boldmath{KMS states  on $C^*$-algebras associated to\\ local homeomorphisms}}
\author{Zahra Afsar}

\author{Astrid an Huef}

\author{Iain Raeburn}

\address{Department of Mathematics and Statistics, University of Otago, PO Box~56, Dunedin 9054, New Zealand}

\email{\{zafsar, astrid, iraeburn\}@maths.otago.ac.nz}

\thanks{This research was supported by the Marsden Fund of the Royal Society of New Zealand.}

\date{21 February 2014}

\begin{document}

\begin{abstract}
For every Hilbert bimodule over a $C^*$-algebra, there are natural gauge actions of the circle on the associated Toeplitz algebra and Cuntz-Pimsner algebra, and hence natural dynamics obtained by lifting these gauge actions to actions of the real line. We study the KMS states of these dynamics for a family of bimodules associated to local homeomorphisms on compact spaces. For inverse temperatures larger than a certain critical value, we find a large simplex of KMS states on the Toeplitz algebra, and we show that all KMS states on the Cuntz-Pimsner algebra have inverse temperature at most this critical value. We illustrate our results by considering the backward shift on the one-sided path space of a finite graph, where we can use recent results about KMS states on graph algebras to see what happens below the critical value. Our results about KMS states on the Cuntz-Pimsner algebra of the shift show that recent constraints on the range of inverse temperatures obtained by Thomsen are sharp.
\end{abstract}

\maketitle

\section{Introduction}

We consider actions $\alpha$ of the real line $\R$ by automorphisms of a $C^*$-algebra $A$. When $\alpha$ describes the time evolution in a model of a physical system, the states of the system are given by positive functionals of norm $1$. The equilibrium states are the states on $A$ that satisfy a commutation relation called the KMS condition. This condition makes sense for every dynamical system of the form $(A,\R,\alpha)$, irrespective of its origin, and studying the KMS states of such systems often yields interesting information. This is certainly the case, for example, for the number-theoretic Hecke algebra of Bost and Connes \cite{BC} and its generalisations \cite{L, LLN}, for systems involving gauge actions on graph algebras \cite{EFW, EL, KW, aHLRS}, and for systems associated to local homeomorphisms of the sort arising in topological dynamics \cite{Th1,Th2}.

Many of the systems studied in the papers mentioned above, and especially those associated to directed graphs, have natural analogues involving Toeplitz algebras in which crucial defining equations are relaxed to inequalities. Work of Exel, Laca and Neshveyev \cite{EL, LN} has shown that there is often a much richer supply of KMS states on these Toeplitz algebras, and this has been extended in recent years to various systems arising in number theory \cite{LR, LN2, CDL}. These papers  contain detailed constructions of the KMS states on the various Toeplitz algebras, and re-examination of the techniques has led to similar constructions in a wide range of examples, including graph algebras \cite{aHLRS, aHLRS2}. In this paper, we use similar techniques to construct KMS states on systems of interest in topological dynamics.

We consider a surjective local homeomorphism $h:Z\to Z$ on a compact Hausdorff space $Z$, and an associated $C^*$-algebra that has been variously described as an Exel crossed product \cite{E}, a groupoid algebra \cite{Th1}, or as both a groupoid algebra and a Cuntz-Pimsner 
algebra \cite{D} (for a precise statement, see \cite[Theorem~3.3]{IM} ). Here we view it as the $C^*$-algebra $\OO(X(E))$ of a topological graph $E$, and then we use the graph-based formalism of Katsura \cite{K} in calculations. The algebra $\OO(X(E))$ carries a canonical gauge action of the circle $\T$, which we lift to an action $\alpha$ of $\R$. We are interested in the KMS states on $(\OO(X(E)),\R,\alpha)$ and its Toeplitz analogue $(\TT(X(E)),\R,\alpha)$. 

Several authors have shown that there is a bijection between the KMS states on $(\OO(X(E)),\R,\alpha)$ and the probability measures on $Z$ that satisfy an invariance relation (for example, \cite[Theorem~9.6]{E} and \cite[Theorem~6.2]{Th1}). To find KMS states, one then has to find invariant measures, and existence has been demonstrated using a functional-analytic analogue of the Perron-Frobenius theory (for example, in \cite[\S6.2]{Th1}). Here we show that, for $\beta$ larger than a critical value $\beta_c$, there is a bijection between the KMS$_\beta$ states on $(\TT(X(E)),\R,\alpha)$ and the probability measures on $Z$ which satisfy an inequality that we call the \emph{subinvariance relation}. We then describe a construction of all the measures satisfying the subinvariance relation, and give a spatial construction of the corresponding KMS states. Putting these constructions together gives a parametrisation of the KMS$_\beta$ states of $(\TT(X(E)),\R,\alpha)$ by a concretely-described simplex of measures on $Z$ for every $\beta>\beta_c$ (Theorem~\ref{th1}). 

Our critical value $\beta_c$ is an exponential bound for the number of preimages of points under iteration of the map $h$, and has previously appeared in the dynamics literature (for example, \cite{FFN,Th1}). In particular, Thomsen has shown that $\beta_c$ is an upper bound for the  inverse temperatures of KMS states on $\OO(X(E))$ \cite[Theorem~6.8]{Th1}. So it seems likely that our results on $\TT(X(E))$ are sharp. At $\beta_c$, we can show by taking limits of states on $\TT(X(E))$ that there exist KMS$_{\beta_c}$ states on $(\OO(X(E)),\alpha)$ (Theorem~\ref{existbetac}).

Our approach is inspired by the analysis of KMS states on the Toeplitz-Cuntz-Krieger algebra $\TT C^*(E)$ of a finite directed graph $E$ in \cite{aHLRS}. The usual description of $C^*(E)$ and $\TT C^*(E)$ using a graph correspondence over the finite-dimensional algebra $C(E^0)$ \cite[\S8]{Ra} does not quite fit our present analysis, though there are striking similarities. However, we can also realise $C^*(E)$ in the present setup as the Cuntz-Pimsner algebra $\OO(X(E^\infty))$ associated to the shift $\sigma$ on the infinite-path space $E^\infty$ \cite[Theorem~5.1]{BRV}. We can therefore test our results by reconciling them with the known results for $C^*(E)$. When $E$ is irreducible in the sense that its vertex matrix $A$ is irreducible, there is a unique KMS state on $(C^*(E),\alpha)$, and its inverse temperature is given in terms of the spectral radius of $A$ by $\beta=\ln \rho(A)$. We confirm that, for the local homeomorphism $\sigma:E^\infty\to E^\infty$, our $\beta_c$ is indeed $\ln\rho(A)$ (Proposition~\ref{compbetac}).

Our computation of $\beta_c$ for shifts works for arbitary matrices of nonnegative integers, so we also consider the reducible case, where there is an interesting variety of examples \cite{aHLRS2}. In \cite[Theorem~6.8]{Th1}, Thomsen also provides a lower bound for the set of possible inverse temperatures of KMS states of $C^*(E)$. The  examples in \cite{aHLRS2} show that Thomsen's bounds are sharp, and that many values in between can be attained as well (see \S\ref{sec:below}). Thus we think that graph algebras could provide an interesting supply of fresh examples for the study of KMS states in dynamics. This should be true also for the study of KMS states on Toeplitz algebras, although there is a curious wrinkle: the Toeplitz algebra $\TT C^*(E)$ embeds in $\TT(X(E^\infty))$, but as a proper subalgebra (see Proposition~\ref{includeTalgs}). Nevertheless, our new results are again compatible with those of \cite{aHLRS, aHLRS2}, and indeed every KMS state of $(\TT C^*(E),\alpha)$ is the restriction of a KMS state of $(\TT(X(E^\infty)),\alpha)$ (Corollary~\ref{allrestrics}).
\smallskip

We begin with a short section on notation and conventions. We then look for a characterisation of KMS states which will allow us to recognise them easily. This characterisation could be of independent interest, because it works for the Toeplitz algebras of quite general Hilbert bimodules (Proposition~\ref{KMScrit}). In \S\ref{sec:subinv}, we discuss our subinvariance relation, which involves a measure-theoretic analogue of a Ruelle operator. Importantly, we describe all solutions of this subinvariance relation (Proposition~\ref{propsR}). In \S\ref{sec:Toeplitz}, we prove our main theorem about KMS states on the Toeplitz algebra, and then in \S\ref{sec:crit} we discuss KMS states at the critical inverse temperature. The last two sections contain our results about shifts on the path spaces of graphs.

\section{Notation and conventions}

\subsection{Toeplitz algebras of Hilbert bimodules} Suppose that $X$ is a \emph{Hilbert bimodule} over a $C^*$-algebra $A$, by which we mean that $X$ is a right Hilbert $A$-module $X$ with a left action of $A$ implemented by a homomorphism $\varphi:A\to \LL(X)$
(in other words, $X$ is a correspondence over $A$). For $m\geq 0$, we write $X^{\otimes m}$ for the internal tensor product $X\otimes _A X\otimes_A\cdots\otimes_A X$ of $m$ copies of $X$, which is also a Hilbert bimodule over $A$. A \emph{representation} $(\psi,\pi)$ of a Hilbert bimodule in a $C^*$-algebra $C$ consists of a linear map $\psi:X\to C$ and a homomorphism $\pi:A\to C$ such that 
\[
\psi(a\cdot x\cdot b)=\pi(a)\psi(x)\pi(b)\text{ and }\pi(\langle x, y\rangle)=\psi(x)^*\psi(y)
\]
for every $x,y\in X$ and $a,b\in B$. For each $m\geq 1$, there is a representation $(\psi^{\otimes m},\pi)$ of $X^{\otimes m}$ such that 
\[
\psi^{\otimes m}(x_1\otimes _A x_2\otimes_A\cdots\otimes_A x_m)=\psi(x_1)\psi(x_2)\cdots\psi(x_m).
\]
For $m=0$, we set $X^{\otimes 0}:=A$ and $\psi^{\otimes 0}:=\pi$.

The \emph{Toeplitz algebra} $\TT(X)$ is generated by a universal representation of $X$, which in this paper we always denote by $(\psi,\pi)$. Proposition~1.3 of \cite{FR} says that there is such an algebra $\TT(X)$, and that it carries a gauge action $\gamma:\T\to \Aut \TT(X)$ characterised by $\gamma_z(\psi(x))=z\psi(x)$ and $\gamma_z(\pi(a))=\pi(a)$. By \cite[Lemma~2.4]{FR}, we have
\[
\TT(X)=\clsp\{\psi^{\otimes m}(x)\psi^{\otimes n}(y)^*:m,n\in\N\}.
\]
If $(\theta,\rho)$ is a representation of $X$ in a $C^*$-algebra $C$, we write $\theta\times \rho$ for the representation of $\TT(X)$ in $C$ such that $(\theta\times \rho)\circ\psi=\theta$ and $(\theta\times \rho)\circ\pi=\rho$.

For $x,y\in X$, we write $\Theta_{x,y}$ for the adjointable operator on $X$ given by $\Theta_{x,y}(z)=x\cdot\langle y,z\rangle$, and $\KK(X):=\clsp\{\Theta_{x,y}:x,y\in X\}\subset \LL(X)$. The representation $(\psi,\pi)$ induces a homomorphism $(\psi,\pi)^{(1)}:\KK(X)\to \TT(X)$ such that $(\psi,\pi)^{(1)}(\Theta_{x,y})=\psi(x)\psi(y)^*$. The \emph{Cuntz-Pimsner algebra} $\OO(X)$ is then the quotient of $\TT(X)$ by the ideal generated by 
\[
\big\{\pi(a)-(\psi,\pi)^{(1)}(\varphi(a)):a\in A\text{ satisfies }\varphi(a)\in \KK(X)\big\}.
\]
(Other definitions of the Cuntz-Pimsner algebra have been used in the literature, but for the bimodules considered here we have $\phi(A)\subset \KK(X)$, and all the definitions give the same algebra.)

\subsection{Measures} We will construct KMS states from Borel measures on compact Hausdorff spaces $Z$. All the measures we consider are regular Borel measures and are positive in the sense that they take values in $[0,\infty)$; indeed, they are all finite measures and hence are automatically regular  (by \cite[Theorem~7.8]{F}, for example). We write $M(Z)_{+}$ for the set of finite Borel measures on $Z$. Some of our measures will be defined by integrals, or  as linear functionals on $C(Z)$, from which the Riesz representation theorem \cite[Corollary~7.6]{F} gives us an (automatically regular) Borel measure.  For us, a \emph{probability measure} is simply a Borel measure with total mass $1$. 

\subsection{Topological graphs} A \emph{topological graph} $E=(E^0,E^1,r,s)$ consists of two locally compact Hausdorff spaces, a continuous map $r:E^1\to E^0$ and a local homeomorphism $s:E^1\to E^0$. For paths in $E$, we use the convention of \cite{Ra}, so that a path of length $2$, for example, is a pair $ef$ with $e,f\in E^1$ and $s(e)=r(f)$. We mention this because in his first paper \cite{K}, Katsura used a different convention, and one has to be careful when consulting the literature because there are other conventions out there. Each such graph $E$ has a Hilbert bimodule $X(E)$ described in \cite[Chapter~9]{Ra}. It is usually a completion of $C_c(E^1)$, but here the spaces $E^0$ and $E^1$ are always compact, and then no completion is necessary because the norm on $X(E)$ is equivalent (as a vector-space norm) to the usual supremum norm on $C(E^1)=X(E)$. For reference, we recall that the module actions are given by $(a\cdot x\cdot b)(z)=a(r(z))x(z)b(s(z))$ and the inner product by $\langle x,y\rangle(z)=\sum_{s(w)=z} \overline{x(w)}y(w)$.

\subsection{KMS states} We use the same conventions for KMS states as other recent papers, such as \cite{LR, LRR, aHLRS}, for example. Suppose that $(A,\R,\alpha)$ is a $C^*$-algebraic dynamical system. An element $a$ of $A$ is \emph{analytic} if $t\mapsto \alpha_t(a)$ is the restriction of an entire function $z\mapsto \alpha_z(a)$ on $\C$. A state $\phi$ of $(A,\R,\alpha)$ is a \emph{KMS state with inverse temperature $\beta$} (or a KMS$_\beta$ state) if $\phi(ab)=\phi(b\alpha_{i\beta}(a))$ for all analytic elements $a,b$. Crucially, it suffices to check this condition for $a,b$ in a family $\FF$ of analytic elements which span a dense subspace of $B$, and it is usually easy to find a good supply of such elements. %For us, KMS$_\infty$ states are limits of KMS$_\beta$ states as $\beta\to \infty$, and ground states are those for which $z\mapsto \phi(a\alpha_z(b))$ is bounded in the upper-half plane for all $a,b\in \FF$.

\section{A characterisation of KMS states}

The following result is similar to \cite[Proposition~2.1(a)]{aHLRS} and \cite[Proposition~4.1]{LRRW}, but is substantially more general. (We have learned that Mitch Hawkins has independently proved a similar result for the bimodules $X(E)$ of topological graphs.)

\begin{prop}\label{KMScrit}
Suppose that $X$ is a Hilbert bimodule over a $C^*$-algebra $A$, and $\alpha:\R\to \Aut A$ is given in terms of the gauge action $\gamma$ by $\alpha_t=\gamma_{e^{it}}$. Suppose $\beta>0$ and $\phi$ is a state on $\TT(X)$. Then $\phi$ is a KMS$_\beta$ state of $(\TT(X),\alpha)$ if and only if $\phi\circ \pi$ is a trace on $A$ and
\begin{equation}\label{commrel}
\phi\big(\psi^{\otimes l}(x)\psi^{\otimes m}(y)^*\big) =  \begin{cases}
0  & \text{ if $m\neq l$}\\
e^{-\beta m}\phi\circ\pi(\langle y,x\rangle_{A}) &  \text{ if  $m=l$.}
\end{cases}
\end{equation}
\end{prop}

\begin{proof}
First suppose that $\phi$ is a KMS$_\beta$ state. For $a\in A$, $\alpha_t(\pi(a))=\pi(a)$ for all $t\in \R$, and hence for all $t\in \C$. Thus the KMS relation says that $\phi\circ\pi$ is a trace. Two applications of the KMS relation give
\begin{align*}
\phi\big(\psi^{\otimes l}(x)\psi^{\otimes m}(y)^*\big)
&=\phi\big(\psi^{\otimes m}(y)^*\alpha_{i\beta}(\psi^{\otimes l}(x))\big)\\
&=e^{-\beta l}\phi\big(\psi^{\otimes m}(y)^*\psi^{\otimes l}(x)\big)\\
&=e^{-\beta (l-m)}\phi\big(\psi^{\otimes l}(x)\psi^{\otimes m}(y)^*\big),
\end{align*}
which because $\beta>0$ implies that both sides vanish for $m\neq l$. Now for $m=l$, the Toeplitz relation for  $(\psi,\pi)$ implies that
\begin{align*}
\phi\big(\psi^{\otimes m}(x)\psi^{\otimes m}(y)^*\big)
=e^{-\beta m}\phi\big(\psi^{\otimes m}(y)^*\psi^{\otimes m}(x)\big)=e^{-\beta m}\phi\big(\pi(\langle y,x\rangle_{A})\big),
\end{align*}
and $\phi$ satisfies \eqref{commrel}.

Next we suppose that $\phi\circ \pi$ is a trace and that $\phi$ satisfies \eqref{commrel}. It suffices for us to prove that
\begin{equation}\label{KMSrel}
\phi(bc)=e^{-\beta(l-m)}\phi(cb)
\end{equation}
for $b=\psi^{\otimes l}(x)\psi^{\otimes m}(y)^*$ and $c=\psi^{\otimes n}(s)\psi^{\otimes p}(t)^*$, where $x$, $y$, $s$ and $t$ are elementary tensors. (When $b$ and/or $c$ lie in $\pi(A)$, this is relatively straightforward because $\phi\circ\pi$ is a trace and $\alpha$ fixes $\pi(A)$.) Formula~\ref{commrel} implies that both sides of \eqref{KMSrel} vanish unless $l+n=m+p$, and hence we assume this from now on. We also assume that $m\leq n$. To see that this suffices, suppose that we have dealt with the case $m\leq n$, and consider $m>n$. Then $\overline{\phi(a)}=\phi(a^*)$ implies that
\begin{align*}
\overline{\phi(bc)}&=\phi(c^*b^*)=\phi\big(\psi^{\otimes p}(t)\psi^{\otimes n}(s)^*\psi^{\otimes m}(y)\psi^{\otimes l}(x)^*\big),
\end{align*}
and we are back in the other case. Thus 
\begin{align*}
\overline{\phi(bc)}&=e^{-\beta(p-n)}\phi\big(\psi^{\otimes m}(y)\psi^{\otimes l}(x)^*\psi^{\otimes p}(t)\psi^{\otimes n}(s)^*\big)\\
&=e^{-\beta(p-n)}\phi(b^*c^*)=\overline{e^{-\beta(p-n)}\phi(cb)};
\end{align*}
since $l+n=m+p$, we have $p-n=l-m$, and we have \eqref{KMSrel}. So it does suffice to prove \eqref{KMSrel} when $m\leq n$.

So we assume that $l+n=m+p$ and $m\leq n$. Then we also have $p\geq l$. Since we are dealing with elementary tensors, we may write $s=s'\otimes s''\in X^{\otimes m}\otimes X^{\otimes(n-m)}$ and $t=t'\otimes t''\in X^{\otimes l}\otimes X^{\otimes(p-l)}$. (If $m=n$ then $p=l$ and we can dispense with this step.) Now we compute, remembering that $p=l+(n-m)$:
\begin{align*}
\phi(bc)&=\phi\big(\psi^{\otimes l}(x)\psi^{\otimes m}(y)^*\psi^{\otimes m}(s')\psi^{\otimes (n-m)}(s'')\psi^{\otimes p}(t)^*\big)\\
&=\phi\big(\psi^{\otimes l}(x)\pi(\langle y,s'\rangle)\psi^{\otimes (n-m)}(s'')\psi^{\otimes p}(t)^*\big)\\
&=\phi\big(\psi^{\otimes l}(x)\psi^{\otimes (n-m)}(\langle y,s'\rangle\cdot s'')\psi^{\otimes p}(t)^*\big)\\
&=e^{-\beta p}\phi\circ \pi\big(\big\langle t'\otimes t'', x\otimes(\langle y,s'\rangle\cdot s'')\big\rangle\big)\quad\text{(using \eqref{commrel}\,)}\\
&=e^{-\beta p}\phi\circ \pi\big(\big\langle t'',\langle t',x\rangle\cdot(\langle y,s'\rangle\cdot s'')\big\rangle\big).
\end{align*}
A similar computation (but using the slightly less obvious identity $\psi(\xi)^*\pi(a)=\psi(a^*\cdot \xi)^*$\,) gives:
\begin{align*}
\phi(cb)&=\phi\big(\psi^{\otimes n}(s)\psi^{\otimes (p-l)}(t'')^*\psi^{\otimes l}(t')^*\psi^{\otimes l}(x)\psi^{\otimes m}(y)^*\big)\\
&=\phi\big(\psi^{\otimes n}(s)\psi^{\otimes (p-l)}(t'')^*\pi(\langle t',x\rangle)\psi^{\otimes m}(y)^*\big)\\
&=\phi\big(\psi^{\otimes n}(s)\psi^{\otimes (p-l)}(\langle x,t'\rangle\cdot t'')^*\psi^{\otimes m}(y)^*\big)\\
&=e^{-\beta n}\phi\circ \pi\big(\big\langle y\otimes(\langle x,t'\rangle\cdot t''),s'\otimes s''\big\rangle\big)\\
&=e^{-\beta n}\phi\circ \pi\big(\big\langle \langle x,t'\rangle\cdot t'',\langle y,s'\rangle\cdot s''\big\rangle\big).
\end{align*}
Since the left action is by adjointable operators, we have
\[
\big\langle \langle x,t'\rangle\cdot t'',\langle y,s'\rangle\cdot s''\big\rangle=\big\langle t'',\langle t',x\rangle\cdot(\langle y,s'\rangle\cdot s'')\big\rangle,
\]
and we deduce from our two calculations that $e^{\beta p}\phi(bc)=e^{\beta n}\phi(cb)$. Since $n-p=m-l$, this is precisely \eqref{KMSrel}.
\end{proof}

\section{KMS states and the subinvariance relation}\label{sec:subinv}

 Suppose $\nu$ is a finite regular Borel measure on a compact Hausdorff space $Z$ and $h:Z\to Z$ is a surjective local homeomorphism. Define $f:C(Z)\rightarrow  \C$ by
 \[
f(a)=\int\sum_{h(w)=z}a(w)\,d\nu(z)\text{ for $a\in C(Z)$.}
\]
Then $f$ is a positive linear functional on $C(Z)$, and hence the Riesz representation theorem (for example, \cite[Theorem~7.2]{F}) says there is a unique finite regular Borel measure $R\nu$ on $Z$ such that
\begin{equation}\label{defR}
\int a\,d(R\nu)=f(a)=\int\sum_{h(w)=z}a(w)\,d\nu(z) \text{ for $a\in C(Z)$.}
\end{equation}
The operation $R$ on measures is affine and positive, and satisfies $\|R\nu\|\leq c_1\|\nu\|$ for the dual norm on $C(Z)^*$, where $c_1:=\max_{z\in Z}|h^{-1}(z)|$. Similar operations appear throughout the analysis of KMS states in dynamics (for example, in \cite[Theorem~6.2]{Th1}), and are sometimes described as ``Ruelle operators''.

\begin{prop}\label{tosubinv}
Suppose that $h:Z\to Z$ is a surjective local homeomorphism on a compact Hausdorff space $Z$. Let $E$ be the topological graph $(Z,Z,\id,h)$ and $X(E)$ the graph correspondence. Define $\alpha:\R\to \Aut \TT (X(E))$ in terms of the gauge action by $\alpha_t=\gamma_{e^{it}}$. Suppose that $\phi$ is a KMS$_\beta$ state on $(\TT({X(E)}),\alpha)$, and $\mu$ is the probability  measure on $Z$ such that $\phi(\pi(a))=\int a\,d\mu$ for all $a\in C(Z)$. Then the measure $R\mu$ satisfies
\begin{align}\label{0}
\int a\,d(R\mu)\leq e^{\beta}\int a\,d\mu\text{ for all positive $a$ in $C(Z)$.}
\end{align}
\end{prop}

\begin{proof}
Suppose that $a\in C(Z)$ and $a\geq 0$. We begin by writing the integrand $\sum_{h(w)=z} a(w)$ in \eqref{defR} in terms of the inner product in $X(E)$. Let $\{U_i\}_{i=0}^k$ be an open cover of $Z$ such that $h|_{U_i}$ is injective, and choose a partition of unity $\{\rho_i\}$ subordinate to $\{U_i\}$. Define $\xi_i\in X(E)$ by $\xi_i=\sqrt{\rho_i}$. Then
\begin{align*}
\sum_{h(w)=z} a(w)\notag&=\sum_{h(w)=z}\sum_{i=0}^k \xi_i(w)^2 a(w)=\sum_{i=0}^k\sum_{h(w)=z} \xi_i(w)^2 a(w)\\
&=\sum_{i=0}^k\sum_{h(w)=z} \overline{\xi_i(w)}(a\cdot\xi_i)(w)
=\sum_{i=0}^k \langle\xi_i,a\cdot\xi_i\rangle(z).
\end{align*}
Thus
\begin{align*}
\int a\,d(R\mu)\notag&=\int \sum_{h(w)=z} a(w)\,d\mu(z)=\int \sum_{i=0}^k\langle\xi_i,a\cdot\xi_i\rangle(z)\,d\mu(z)\\
\notag&=\phi \Big(\pi\Big(\sum_{i=0}^k\langle\xi_i,a\cdot\xi_i\rangle\Big)\Big)=\sum_{i=0}^k \phi(\psi(\xi_i)^*\psi(a\cdot\xi_i)).
\end{align*}
Now, since $\phi$ is a KMS$_\beta$ state, we have
\begin{align}\label{equ40}
\int a\,d(R\mu)=\sum_{i=0}^k e^{\beta}\phi(\psi(a\cdot\xi_i)\psi(\xi_i)^*).
\end{align}

Our next task is to compare the operator $\sum_{i=0}^k \psi(a\cdot\xi_i)\psi(\xi_i)^*$ appearing on the right-hand side of \eqref{equ40} with $\pi(a)$. For this, we use the Fock representation $(T,\varphi_\infty)$ of $\TT(X(E))$ from \cite[Example~1.4]{FR}. As a right $A$-module, $F(X(E))$ is the Hilbert module direct sum $\bigoplus_{n=0}^\infty X(E)^{\otimes n}$, with the left action of $A$ by diagonal operators giving a homomorphism $\varphi_\infty:A\to \LL(F(X(E)))$. The homomorphism $T:X(E)\to \LL(F(X(E)))$ sends $x\in X(E)$ to the creation operator $T(x):y\mapsto x\otimes_A y$, and $T\times \varphi_\infty$ is an injection on $\TT(X(E))$ \cite[Corollary~2.2]{FR}.

Let $n\geq 1$ and $x=x_1\otimes\dots\otimes x_n\in  X(E)^{\otimes n}$. Then
\begin{align*}\label{equ41}
\sum_{i=0}^k T( a\cdot\xi_i)T(\xi_i)^*(x)\notag&=\sum_{i=0}^k T( a\cdot\xi_i)\big(\langle\xi_i,x_1\rangle\cdot x_2\otimes\dots\otimes x_n\big)\\
&=\sum_{i=0}^k  \big(a\cdot\xi_i\cdot\langle\xi_i,x_1\rangle\big)\otimes x_2\otimes\dots\otimes x_n.
\end{align*}
Since $h|_{U_i}$ is injective and $\supp\xi_i\subset U_i,$ we have
\[
(\xi_i\cdot\langle\xi_i,x_1\rangle)(z)=\xi_i(z)\langle\xi_i,x_1\rangle(h(z))\notag=\xi_i(z)\sum_{h(w)=h(z)} \overline{\xi_i(w)}x_1(w)=\xi_i(z)^2x_1(z).
\]
Thus 
\begin{align*}\label{equ43}
\sum_{i=0}^k T( a\cdot\xi_i)T(\xi_i)^*(x)\notag&=\sum_{i=0}^k a\cdot(\xi_i^2x_1\otimes x_2\otimes\dots\otimes x_n)=a\cdot x=\varphi_\infty(a)(x).
\end{align*}
Thus $\sum_{i=0}^k T( a\cdot\xi_i)T(\xi_i)^* = \varphi_\infty(a)$ as operators on   $X(E)^{\otimes n}$ for $n\geq 1$. Since each $T( a\cdot\xi_i)T(\xi_i)^*$ vanishes on $C(Z)=X(E)^{\otimes 0}$ and $a$ is positive, we have
\begin{equation*}
\sum_{i=0}^k T( a\cdot\xi_i)T(\xi_i)^* \leq \varphi_\infty(a)\quad\text{in $\LL(F(X(E)))$;}
\end{equation*}
since the homomorphism $T\times\varphi_\infty$ is faithful, we deduce that
\begin{equation}\label{ineqinTT}
\sum_{i=0}^k \psi( a\cdot\xi_i)\psi(\xi_i)^* \leq \pi(a)\quad\text{in $\TT(X(E))$.}
\end{equation}

To finish off, we apply $\phi$ to \eqref{ineqinTT}:
\begin{equation*}\label{2}
\phi\Big(\sum_{i=0}^k \psi( a\cdot\xi_i)\psi(\xi_i)^*\Big) \leq \phi(\pi(a))=\int a\,d\mu. 
\end{equation*}
On the other hand, \eqref{equ40} implies that
\[
\phi\Big(\sum_{i=0}^k \psi( a\cdot\xi_i)\psi(\xi_i)^*\Big)=e^{-\beta}\int a\,d(R\mu),
\]
and the result follows from the last two displays.
\end{proof}

When $Z$ is a finite set and $A$ is a nonnegative matrix, $\mu$ is a vector in $[0,\infty)^{Z}$, and the relation \eqref{0} in the form $A\mu\leq e^{\beta}\mu$ says that $\mu$ is a subinvariant vector for $A$ in the sense of Perron-Frobenius theory. Subinvariant vectors played an important role in the analysis of KMS states on the Toeplitz algebras of graphs in \cite[\S2]{aHLRS}, and \eqref{0} will play a similar role in our analysis. So we shall refer to \eqref{0} as the \emph{subinvariance relation}. 

We now show how to construct the probability measures which satisfy the subinvariance relation. Proposition~\ref{propsR} is an analogue for our operation $R$ on measures of \cite[Theorem~3.1(a)]{aHLRS}, which is about the subinvariance relation for the vertex matrix of a finite directed graph. Here the powers $R^n$ are defined inductively by $R^{n+1}\nu=R(R^n\nu)$, and then we have
\begin{equation}\label{Rn}
\int a\,d(R^n\nu)=\int\sum_{h^n(w)=z}a(w)\,d\nu(z)\quad\text{for $a\in C(Z)$.}
\end{equation}

\begin{prop}\label{propsR} 
Suppose that $h:Z\to Z$ is a surjective local homeomorphism on a compact Hausdorff space $Z$. Let
\begin{equation}\label{defcrit}
\beta_c:=\limsup_{n\to \infty}\Big(n^{-1}\ln\Big(\max_{z\in Z}|h^{-n}(z)|\Big)\Big),
\end{equation}
and suppose that $\beta>\beta_c$.
\begin{enumerate}
\item\label{fbeta} The series $\sum_{n=0}^\infty e^{-\beta n}|h^{-n}(z)|$ converges uniformly for $z\in Z$ to a continuous function $f_\beta(z)$, which satisfies
\begin{equation}\label{lowerbd}
f_\beta(z)-\sum_{h(w)=z}e^{-\beta}f_\beta(w)= 1\quad\text{for all $z\in Z$.}
\end{equation}
\item\label{resolvent} Suppose that $\varepsilon$ is a finite regular Borel measure on $Z$. Then the series $\sum_{n=0}^\infty e^{-\beta n}R^n\varepsilon$ converges in norm in the dual space $C(Z)^*$ with sum $\mu$, say. Then $\mu$ satisfies the subinvariance relation \eqref{0}, and we have $\varepsilon=\mu-e^{-\beta}R\mu$. Then $\mu$ is a probability measure if and only if $\int f_\beta
\,d\varepsilon=1$.
\item\label{positivemass} 
Suppose that $\mu$ is a probability measure which satisfies the subinvariance relation \eqref{0}. Then $\varepsilon=\mu-e^{-\beta}R\mu$ is a finite regular Borel measure satisfying $\int f_\beta\,d\varepsilon=1$, and we have $\mu=\sum_{n=0}^\infty e^{-\beta n}R^n\varepsilon$.
\end{enumerate}
\end{prop}

\begin{remark}
Part~\eqref{resolvent} applies when $\epsilon=0$, and gives $\mu=0$. However, it is implicit in part~\eqref{positivemass} that $\epsilon$ is not zero (because $\int f_\beta\,d\varepsilon=1$), and hence $\mu\not=e^{-\beta}R\mu$. Thus part~\eqref{positivemass} implies that the invariance relation $R\mu=e^{\beta}\mu$ has no solutions\footnote{The analogue of Proposition~\ref{tosubinv} for the Cuntz-Pimsner algebra will say that the measure $\mu$ satisfies the invariance relation. Thus Proposition~\ref{propsR}\eqref{positivemass} will imply that there are no KMS$_\beta$ states on $\OO(X(E))$ for $\beta> \beta_c$. This is consistent with \cite[Theorem~6.8]{Th1} and our Corollary~\ref{crittemp}.} for $\beta>\beta_c$.
\end{remark}

\begin{proof}
We first claim that there exist $\delta>0$ and $K\in \N$ such that
\begin{equation}\label{estinvimage}
m\geq K\Longrightarrow e^{-\beta m}|h^{-m}(z)|< e^{-\delta m}\quad\text{for all $z\in Z$.}
\end{equation}
Write $c_n:=\max_{z\in Z}|h^{-n}(z)|$, so that $\beta>\beta_c$ means $\beta>\limsup n^{-1}\ln c_n$. Then for large $n$, we have $\beta>\sup_{m\geq n}m^{-1}\ln c_m$. Thus there exist $\delta>0$ and $K$ such that
\begin{align*}
m\geq K&\Longrightarrow \beta-\delta>m^{-1}\ln c_m\Longrightarrow c_m< e^{\beta m-\delta m}\\
&\Longrightarrow e^{-\beta m}|h^{-m}(z)|< e^{-\delta m}\quad\text{for all $z\in Z$.}
\end{align*}
This proves our claim.

Take $\delta$ as in \eqref{estinvimage}.  Then comparing the series $\sum e^{-\beta n}|h^{-n}(z)|$ with $\sum e^{-\delta n}$ shows that the series $\sum_{n=0}^\infty e^{-\beta n}|h^{-n}(z)|$ converges uniformly for $z\in Z$. Since $h$ is a local homeomorphism on a compact space, each $z\mapsto |h^{-1}(z)|$ is locally constant (by \cite[Lemma~2.2]{BRV}, for example), and hence continuous. Thus $f_\beta(z):=\sum_{n=0}^\infty e^{-\beta n}|h^{-n}(z)|$ is the uniform limit of a sequence of continuous functions, and is therefore continuous. To see \eqref{lowerbd}, we note that because all the series converge absolutely, we can interchange the order of sums in the following calculation: 
\begin{align*}
f_\beta(z)-\sum_{h(w)=z}&e^{-\beta}f_\beta(w)\\
&=\sum_{n=0}^\infty e^{-\beta n}|h^{-n}(z)|-\sum_{h(w)=z}e^{-\beta}\Big(\sum_{m=0}^\infty e^{-\beta m}|h^{-m}(w)|\Big)\\
&=\sum_{n=0}^\infty e^{-\beta n}|h^{-n}(z)|-\sum_{m=0}^\infty e^{-\beta (m+1)}\Big(\sum_{h(w)=z}|h^{-m}(w)|\Big)\\
&=\sum_{n=0}^\infty e^{-\beta n}|h^{-n}(z)|-\sum_{m=0}^\infty e^{-\beta (m+1)}|h^{-(m+1)}(z)|\\
&=e^{-\beta 0}|h^{-0}(z)|=1.
\end{align*}
We have now proved \eqref{fbeta}.

Next, we look at the series in \eqref{resolvent}. Take $\delta, K$ satisfying \eqref{estinvimage}. Then for $N>M\geq K$ and $g\in C(Z)$ we calculate using \eqref{Rn}:
\begin{align*}
\Big|\sum_{n=M+1}^N e^{-\beta n}\int g\,d(R^n\varepsilon)\Big|
&=\Big|\sum_{n=M+1}^N e^{-\beta n}\int \sum_{h^n(w)=z}g(w)\,d\varepsilon(z)\Big|\\
&\leq \sum_{n=M+1}^N e^{-\beta n} |h^{-n}(z)|\,\|\varepsilon\|_{C(Z)^*}\|g\|_\infty\\
&\leq \sum_{n=M+1}^N e^{-\delta n}\|\varepsilon\|_{C(Z)^*}\|g\|_\infty.
\end{align*}
Thus the series $\sum_{n=0}^\infty e^{-\beta n}R^n\varepsilon$ converges in the norm of $C(Z)^*$, as asserted in \eqref{resolvent}. Since the operation $R$ is affine and norm-continuous on positive measures, the sum $\mu:=\sum_{n=0}^\infty e^{-\beta n}R^n\varepsilon$ satisfies
\[
\mu-e^{-\beta}R\mu=\sum_{n=0}^\infty e^{-\beta n}R^n\varepsilon-\sum_{n=0}^\infty e^{-\beta (n+1)}R^{n+1}\varepsilon=\varepsilon;
\]
since $\varepsilon$ is a (positive) measure, this implies that $\mu$ satisfies the subinvariance relation. The Riesz representation theorem implies that $\mu$ is a regular Borel measure, and
\begin{align*}
\mu(Z)&=\sum_{n=0}^\infty e^{-\beta n}(R^n\varepsilon)(Z)=\sum_{n=0}^\infty e^{-\beta n}\int 1\,d(R^n\varepsilon)\\
&=\sum_{n=0}^\infty e^{-\beta n}\int |h^{-n}(z)|\,d\varepsilon(z),
\end{align*}
which by the monotone convergence theorem is $\int f_\beta\,d\varepsilon$. Thus $\mu$ is finite, and it is a probability measure if and only if $\int f_\beta\,d\varepsilon=1$.

For part \eqref{positivemass}, we first note that the subinvariance relation implies that $\varepsilon$ is a positive measure, and it is finite because $\mu$ is. Next we compute:
\begin{align*}
\int f_\beta\,d\varepsilon
&=\int f_\beta\,d\mu -e^{-\beta}\int f_\beta\,d(R\mu)\\
&=\int f_\beta(z)\,d\mu(z) -e^{-\beta}\int \sum_{h(w)=z}f_\beta(w)\,d\mu(z)\\
&=\int \Big(f_\beta(z)-\sum_{h(w)=z}e^{-\beta}f_\beta(w)\Big)\,d\mu(z),
\end{align*}
which by \eqref{lowerbd} is $\mu(Z)=1$. Finally, we have
\begin{align*}
\sum_{n=0}^\infty e^{-\beta n}R^n\varepsilon&=\sum_{n=0}^\infty e^{-\beta n}R^n(\mu-e^{-\beta}R\mu)\\
&=\sum_{n=0}^\infty e^{-\beta n}R^n\mu-\sum_{n=0}^\infty e^{-\beta (n+1)}R^{n+1}\mu=\mu.\qedhere
\end{align*}
\end{proof}

\section{KMS states on the Toeplitz algebra}\label{sec:Toeplitz}

Our main theorem is the following analogue of \cite[Theorem~3.1]{aHLRS}.

\begin{thm}\label{th1}
Suppose that $h:Z\to Z$ is a surjective local homeomorphism on a compact Hausdorff space $Z$, $E$ is the topological graph $(Z,Z,\id,h)$, and $X(E)$ is the graph correspondence. Define $\alpha:\R\to \Aut \TT (X(E))$ in terms of the gauge action by $\alpha_t=\gamma_{e^{it}}$. Take $\beta_c$ as in \eqref{defcrit}, suppose that $\beta>\beta_c$, and let $f_\beta$ be the function in Proposition~\ref{propsR}\,\eqref{fbeta}.
\begin{enumerate}
\item\label{3a}  Suppose that $\varepsilon$ is a finite regular Borel measure on $Z$ such that $\int f_\beta\,d\varepsilon=1$, and take $\mu=\sum_{n=0}^\infty e^{-\beta n}R^n\varepsilon$. Then there is a KMS$_\beta$ state $\phi_\varepsilon$ on $(\TT(X(E)),\alpha)$ such that
\begin{equation}\label{equ2}
\phi_{\varepsilon}\big(\psi^{\otimes l}(x)\psi^{\otimes m}(y)^*\big) =  \begin{cases}
0  & \text{if $l\neq m$}\\
e^{-\beta m}\int\langle y,x\rangle\,d\mu &  \text{if  $l=m$.}
\end{cases}
\end{equation}
\item\label{3b} The map $\varepsilon\mapsto \phi_\varepsilon$ is an affine isomorphism of 
\[\
\Sigma_\beta:=\Big\{\varepsilon \in M(Z)_+: \int f_\beta\,d\varepsilon=1\Big\}
\]
onto the simplex of $KMS_\beta$ states of $(\TT(X(E)),\alpha)$. The inverse takes a state $\phi$ to $\varepsilon:=\mu-e^{-\beta}R\mu$, where $\mu$ is the probability measure such that $\phi(\pi(a))=\int a\,d\mu$ for $a\in C(Z)$.
\end{enumerate}
\end{thm}

In the proof of this theorem, we will need to do some computations in the Toeplitz algebra, and the following observation will help.

\begin{lemma}\label{idtps}
For $n\geq 1$ we consider the topological graph $F_n=(Z,Z,\id,h^n)$. Then there is an isomorphism $\rho_n$ of $X(E)^{\otimes n}$ onto $X(F_n)$ such that
\[
\rho_n(x_1\otimes x_2\otimes\cdots\otimes x_n)(z)=x_1(z)x_2(h(z))\cdots x_n(h^{n-1}(z)).
\]  
\end{lemma}

\begin{proof}
We prove this by induction on $n$. It is trivially true for $n=1$ --- indeed, we have $E=F_1$, and $\rho_1$ is the identity. Suppose that there is such an isomorphism $\rho_n$ and define $\rho_{n+1}(x_1\otimes x)(z)=x_1(z)\rho_n(x)(h(z))$. Routine calculations show that $\rho_{n+1}$ is a bimodule homomorphism. We next show that $\rho_{n+1}$ preserves the inner products. Let $x_1\otimes x$ and $y_1\otimes y$ be elementary tensors in $X(E)\otimes_{C(Z)} X(E)^{\otimes n}$. Then for $z\in Z$ we have
\begin{align*}
\big\langle \rho_{n+1}(x_1\otimes x),\,&\rho_{n+1}(y_1\otimes y)\big\rangle(z)
=\sum_{h^{n+1}(w)=z}\overline{x_1(w)\rho_n(x)(h(w))}y_1(w)\rho_n(y)(h(w))\\
&=\sum_{h^{n}(v)=z}\sum_{h(w)=v}\overline{x_1(w)\rho_n(x)(h(w))}y_1(w)\rho_n(y)(h(w))\\
&=\sum_{h^{n}(v)=z}\overline{\rho_n(x)(v)}\Big(\sum_{h(w)=v}\overline{x_1(w)}y_1(w)\Big)\rho_n(y)(v)\\
&=\sum_{h^{n}(v)=z}\overline{\rho_n(x)(v)}\big(\langle x_1,y_1\rangle\cdot\rho_n(y))(v)\\
&=\big\langle \rho_n(x),\langle x_1,y_1\rangle\cdot\rho_n(y)\big\rangle(z)\\
&=\big\langle x_1\otimes\rho_n(x),y_1\otimes\rho_n(y)\big\rangle(z).
\end{align*}
Since the range of $\rho_{n+1}$ contains $C(Z)$ (take $x=1$), we deduce that $\rho_{n+1}$ is an isomorphism of Hilbert bimodules.
\end{proof}

\begin{proof}[Proof of Theorem~\ref{th1}]
We aim to construct the KMS state $\phi_\varepsilon$ using a representation\footnote{As in our previous papers, this construction was motivated by the one in the proof of \cite[Theorem~1.4]{LN}, which suggests that we should take a representation, here the representation $M_\varepsilon$ of $A=C(Z)$ by multiplication operators on $L^2(Z,\varepsilon)$, and work in the induced representation $F(X(E))\dashind_A^{\TT} M_\varepsilon$ of $\TT=\TT(X(E))$, where $F(X(E))$ is the Fock bimodule. However, this requires many identifications, and it seems clearer to write down a concrete Hilbert space.} $(\theta,\rho)$ of $X(E)$ on $H_{\theta,\rho}:=\bigoplus_{n=0}^\infty L^2(Z,R^n\varepsilon)$. We write elements of the direct sum as sequences $\xi=(\xi_n)$. For $a\in C(Z)$, we take $\rho$ to be the direct sum of the representations $\rho_n$ of $C(Z)$ on $L^2(Z,R^n\varepsilon)$ given by $(\rho_n(a)\xi_n)(z)=a(z)\xi_n(z)$. Next we claim that for each $x\in X$ there is a bounded operator $\theta(x)$ on $H_{\theta,\rho}$ such that
\[
(\theta(x)\xi)_{n+1}(z)=x(z)\xi_n(h(z))\text{ for $n\geq 0$ and } (\theta(x)\xi)_{0}=0.
\]
To justify the claim, we take  $\xi=(\xi_n)\in\bigoplus_{n=0} ^ \infty  L^2(Z, R^n\varepsilon)$ and compute:
\begin{align*}
\|\theta(x)\xi\|^2&=\sum_{n=0}^\infty \|(\theta(x)\xi_n)_{n+1}\|^2\\
&= \sum_{n=0}^\infty \int |x(z)|^2|\xi_n(h(z))|^2\,d(R^{n+1}\varepsilon)(z)\\
&\leq\sum_{n=0}^\infty \|x\|_\infty^2\int \sum_{h(w)=z}|\xi_n(h(w))|^2\,d(R^{n}\varepsilon)(z)\\
&=\sum_{n=0}^\infty \|x\|_\infty^2\int \sum_{h(w)=z}|\xi_n(z)|^2\,d(R^{n}\varepsilon)(z)\\
&\leq\sum_{n=0}^\infty \|x\|_\infty^2 c_1 \int |\xi_n(z)|^2\,d(R^{n}\varepsilon)(z)\quad\text{(where $c_1=\textstyle{\max_z|h^{-1}(z)|}$)}  \\
&=c_1\|x\|_\infty^2 \|\xi\|^2.
\end{align*}
A similar calculation shows that the adjoint $\theta(x)^*$ satisfies
\begin{equation}\label{equ44}
(\theta(x)^*\eta)_n(z)=\sum_{h(w)=z}\overline{x(w)}\eta_{n+1}(w)\quad \text{for $\eta\in H_{\theta,\rho}$.}
\end{equation}

Next we claim that $(\theta,\rho)$ is a representation of $X(E)$. It is easy to check that $\theta(a\cdot x)=\rho(a)\theta(x)$, and almost as easy to see that $\theta(x\cdot a)=\theta(x)\rho(a)$: for $\xi=(\xi_n)$ we have $(\theta(x\cdot a)\xi)_{0}=0=(\theta(x)(\rho(a)\xi))_0$, and for $n\geq 1$
\begin{align*}
(\theta(x\cdot a)\xi)_{n}(z)&=(x(z)a(h(z)))\xi_{n-1}(h(z))=x(z)(\rho(a)\xi)_{n-1}(h(z))\\
&=\big(\theta(x)(\rho(a)\xi)\big)_{n}(z).
\end{align*}
For $n\geq 0$, we have
\begin{align*}
(\rho(\langle x,y\rangle)\xi)_n(z)&=\langle x,y\rangle(z) \xi_n(z)=\sum_{h(w)=z}\overline{x(w)}y(w)\xi_n(z)\\
&=\sum_{h(w)=z}\overline{x(w)}y(w)\xi_n(h(w))=\sum_{h(w)=z}\overline{x(w)}(\theta(y)\xi)_{n+1}(w)\\
&=(\theta(x)^*\theta(y)\xi)_n(z)\quad\text{(using \eqref{equ44}).}
\end{align*}
Now the universal property of $\TT(X(E))$ gives a homomorphism $\theta\times \rho:\TT(X(E))\to B(H_{\theta,\rho})$ such that $(\theta\times \rho)\circ\psi=\theta$ and $(\theta\times \rho)\circ\pi=\rho$.

For each $k\geq 1$ we choose a finite partition $\{Z_{k,i}: 1\leq i\leq I_k\}$ of $Z$ by Borel sets such that $h^k$ is one-to-one on each $Z_{k,i}$. We write also $I_0=1$ and $Z_{0,1}=Z$. Let $\chi_{k,i}=\chi_{Z_{k,i}}$, and define $\xi^{k,i}\in \bigoplus_{n=0}^\infty L^2(Z,R^n\varepsilon)$ by
\begin{equation*}
\xi_n^{k,i}= \begin{cases}
0  & \text{if $n\neq k$}\\
\chi_{k,i}& \text{if $n=k$.}
\end{cases}
\end{equation*}
We aim to define our state $\phi_{\varepsilon}:\TT(X(E))\to \C$ by
\begin{equation}\label{defphi}
\phi_{\varepsilon}(b)=\sum_{k=0}^\infty\sum_{i=1}^{I_k} e^{-\beta k}\big(\theta\times\rho(b)\xi^{k,i}\,|\,\xi^{k,i}\big)\quad\text{for  $b\in \TT(X(E))$,}
\end{equation}
but of course we have to show that the series converges. It suffices to do this for positive $b$, and then since $b\leq \|b\|1$ it suffices to prove that the series for $\phi_{\varepsilon}(1)$ converges. Since for each $k$ the $Z_{k,i}$ partition $Z$, we have
\[
\sum_{k=0}^\infty\sum_{i=1}^{I_k}e^{-\beta k}\big(\chi_{Z_{k,i}}\,|\,\chi_{Z_{k,i}}\big)
=\sum_{k=0}^\infty\sum_{i=1}^{I_k}e^{-\beta k}R^k\varepsilon(Z_{k,i})
=\sum_{k=0}^\infty e^{-\beta k} R^k\varepsilon(Z);
\]
Proposition~\ref{propsR} implies that this converges with sum $\mu(Z)=1$. Thus the formula \eqref{defphi} gives us a well-defined state on $\TT(X(E))$. 

We now prove that this state satisfies \eqref{equ2}. 
So we take $x\in X(F_l)=X^{\otimes l}$,  $y\in X(F_m)=X^{\otimes m}$ and $b=\psi^{\otimes l}(x)\psi^{\otimes m}(y)^*$. Since $\xi^{k,i}$ is zero in all except the $k$th summand of $\bigoplus_{n=0}^\infty L^2(Z,R^n\varepsilon)$,
\[
\theta\times\rho(b)\xi^{k,i}=\theta^{\otimes l}(x)\theta^{\otimes m}(y)^*\xi^{k,i}
\]
is zero in all but the $(k-m+l)$th summand. Thus
\[
(\theta\times\rho(b)\xi^{k,i}\,|\,\xi^{k,i})=0\quad \text{for all $k,i$ whenever $l\not=m$,}
\]
and $\phi_{\varepsilon}$ certainly satisfies \eqref{equ2} when $l\not=m$. So we suppose that $l=m\geq 0$. 

Next, note that $\theta^{\otimes m}(x)\theta^{\otimes m}(y)^*\xi^{k,i}=0$ if $k<m$. For $k\geq m$, we know that $h^k$ is injective on $Z_{k,i}$, and hence so is $h^m$. Thus $w,z\in Z_{k,i}$ and $h^m(w)=h^m(z)$  imply $w=z$, and
\begin{align*}
\big(\theta^{\otimes m}(x)\theta^{\otimes m}(y)^*\xi^{k,i}\,|\,\xi^{k,i}\big)
&=\int \Big(x(z)\sum_{h^m(w)=h^m(z)}\overline{y(w)}\chi_{k,i}(w)\Big)\overline{\chi_{k,i}(z)}\,d(R^k\varepsilon)(z)\\
&=\int x(z)\overline{y(z)}\overline{\chi_{k,i}(z)}\,d(R^k\varepsilon)(z).
\end{align*}
Since the $Z_{k,i}$ partition $Z$, summing over $i$ gives
\[
\sum_{i=1}^{I_k} \big(\theta\times\rho(\psi^{\otimes m}(x)\psi^{\otimes m}(y)^*)\xi^{k,i}\,|\,\xi^{k,i}\big)
=\int x(z)\overline{y(z)}\,d(R^k\varepsilon)(z).
\]
 Thus from \eqref{Rn} and the formula for the inner product on $X(E)^{\otimes m}=X(F_m)$ we have
\begin{align}\label{CalcKMS}
\phi_{\varepsilon}\big(\psi^{\otimes m}&(x)\psi^{\otimes m}(y)^*\big)=\sum_{k=m}^\infty e^{-\beta k}\int x(z)\overline{y(z)}\,d(R^k\varepsilon)(z)\\
\notag&=\sum_{k=m}^\infty  e^{-\beta k}\int\sum_{h^m(w)=z} x(w)\overline{y(w)}\,d(R^{k-m}\varepsilon)(z)\\
\notag&=\sum_{k=0}^\infty  e^{-\beta (m+k)}\int\langle y, x\rangle(z)\,d(R^{k}\varepsilon)(z)\\
\notag&=e^{-\beta m}\int \langle y, x\rangle\,d\Big(\sum_{k=0}^\infty  e^{-\beta k}R^{k}\varepsilon\Big)\quad \text{by Proposition~\ref{propsR}\eqref{resolvent}}\\
\notag&=e^{-\beta m}\int\langle y, x\rangle \,d\mu.
\end{align}
This is \eqref{equ2}. Applying \eqref{equ2} with $m=0$ shows that $\phi_{\varepsilon}(\pi(a))=\int a\,d\mu$, which says that the last integral in \eqref{CalcKMS} is $\phi_{\varepsilon}\circ\pi (\langle y,x\rangle)$. Thus $\phi_{\varepsilon}$ satisfies \eqref{commrel}, and Proposition~\ref{KMScrit} implies that $\phi_{\varepsilon}$ is a KMS$_\beta$ state. We have now proved part~\eqref{3a}.

Now suppose that $\phi$ is a KMS$_\beta$ state, and let $\mu$ be the probability measure such that $\phi\circ\pi(a)=\int a\,d\mu$ for $a\in C(Z)$. Then Proposition~\ref{tosubinv} implies that $\mu$ satisfies the subinvariance relation $R\mu\leq e^{\beta}\mu$, and hence Proposition~\ref{propsR}\eqref{positivemass} implies that $\varepsilon:=\mu-e^{-\beta}R\mu$ is a positive measure which belongs to $\Sigma_\beta$ and satisfies $(1-e^{-\beta}R)^{-1}\varepsilon=\mu$. Thus formulas~\eqref{commrel} and~\eqref{equ2} imply that $\phi=\phi_\varepsilon$. This shows that $\varepsilon\mapsto \phi_\epsilon$ is surjective. Since applying the construction of this paragraph to the state $\phi_{\varepsilon}$ gives us $\epsilon=\mu-e^{-\beta} R\mu$ back, it also shows that $\varepsilon\mapsto \phi_{\varepsilon}$ is one-to-one.

Thus $\varepsilon\mapsto \phi_\varepsilon$ maps $\Sigma_\beta$ onto the set of KMS$_\beta$ states, and it is affine and continuous for the respective weak* topologies. So we have proved our theorem.
\end{proof}

The next Corollary is contained in \cite[Theorem~6.8]{Th1} (here the function $F$ of that theorem is identically $1$ --- see Remark~\ref{compwithT} below), but the proof in \cite{Th1} is quite different.

\begin{cor}\label{crittemp}
Take $h:Z\to Z$ and $E$ as in Theorem~\ref{th1}, and define $\alpha:\R\to \Aut \OO(X(E))$ in terms of the gauge action $\gamma$ by $\alpha_t=\gamma_{e^{it}}$. If there is a KMS state of $(\OO(X(E)),\alpha)$ with inverse temperature $\beta$, then $\beta\leq \beta_c$.
\end{cor}

\begin{proof}
Suppose $\beta>\beta_c$ and there is a KMS$_\beta$ state $\phi$ of $(\OO(X(E)),\alpha)$. Denote by $q$ the quotient map of $\TT(X(E))$ onto $\OO(X(E))$. Then $\phi\circ q$ is a KMS$_\beta$ state of the system $(\TT(X(E)),\alpha)$ considered in Theorem~\ref{th1}. Thus there is a measure $\varepsilon$ on $Z$ such that $\int f_\beta\,d\varepsilon=1$ and $\phi\circ q=\phi_\varepsilon$. Notice in particular that $\varepsilon(Z)>0$. We can find a finite open cover $\{U_j:1\leq j\leq I\}$ of $Z$ by sets such that $h|_{U_j}$ is a homeomorphism, and we can find open sets $\{V_j:1\leq j\leq I\}$ which still cover $Z$ but have $\overline{V_j}\subset U_j$ (see \cite[Lemma~4.32]{tfb}, for example). Since $\varepsilon(Z)>0$, there exists $j$ such that $\varepsilon(V_j)>0$. Now choose a function $f\in C_c(Z)$ such that $f(z)\not=0$ for $z\in V_j$ and $\supp f\subset U_j$. Then the left action of $|f|^2\in C(Z)$ on $X(E)$ is implemented by the finite-rank operator $\Theta_{f,f}$, and hence 
\begin{align*}
\pi(|f|^2)-\psi(f)\psi(f)^*&=\pi(|f|^2)-(\psi,\pi)^{(1)}(\Theta_{f,f})\\
&=\pi(|f|^2)-(\psi,\pi)^{(1)}(\varphi(|f|^2))
\end{align*}
belongs to the kernel of the quotient map $q$. But with $\mu$ as in Theorem~\ref{th1}\eqref{3b}, we have
\begin{align*}
\phi_\varepsilon(\pi(|f|^2)-\psi(f)\psi(f)^*)
&=\int |f|^2\,d\mu-e^{-\beta}\int \sum_{h(w)=z}|f|^2(w)\,d\mu\\
&=\int |f|^2\,d(\mu-e^{-\beta}R\mu)=\int |f|^2\,d\varepsilon>0.
\end{align*}
Thus $\phi_\varepsilon$ does not vanish on $\ker q$, and we have a contradiction. Thus $\beta\leq \beta_c$.
\end{proof}

\begin{example}
Suppose that $A\in M_d(\Z)$ is an integer matrix with $N:=|\det A|>1$. Then there is a covering map $\sigma_A:\T^d\to \T^d$ such that $\sigma_A(e^{2\pi ix})=e^{2\pi iAx}$ for $x\in \R^d$. The inverse image of each $z\in \T^d$ has $N$ elements, and hence $|\sigma_A^{-n}(z)|=N^n$ for all $z$. Thus
\[
\frac{1}{n}\ln\Big(\max_{z\in \T^d}|\sigma_A^{-n}(z)|\Big)=\frac{1}{n}\ln N^n=\ln N\quad\text{for all $n$,}
\]
and $\beta_c=\ln N$. Suppose $\beta>\ln N$ and $\nu$ is a probability measure on $\T^d$. The function $f_\beta$ is the constant function
\[
f_\beta\equiv \sum_{n=0}^\infty e^{-\beta n}N^n=\frac{1}{1-Ne^{-\beta}},
\]
and hence the measure $\varepsilon:=(1-Ne^{-\beta})\nu$ satisfies $\int f_\beta\,d\varepsilon=1$. Thus with $E=(\T^d,\T^d,\id,\sigma_A)$, Theorem~\ref{th1} gives a KMS$_\beta$ state $\phi_\varepsilon$ on $(\TT(X(E)),\alpha)$ such that
\begin{equation}\label{KMSdilation}
\phi_{\varepsilon}\big(\psi^{\otimes k}(x)\psi^{\otimes l}(y)^*\big) = \delta_{k,l}
e^{-\beta k}\sum_{j=0}^\infty e^{-\beta j}\int\langle y,x\rangle\,d(R^j\varepsilon)
\end{equation}
for $x\in X^{\otimes k}$, $y\in X^{\otimes l}$. We claim that $\phi_{\varepsilon}$ is the KMS state $\psi_{\beta,\nu}$ described in \cite[Proposition~6.1]{LRR}.

The algebra $\TT(M_L)$ in \cite{LRR} is associated to an Exel system $(C(\T^d),\sigma_A^*,L)$, in which $\sigma_A^*$ is the endomorphism $f\mapsto f\circ \sigma_A$ and $L$ is a ``transfer operator'' defined by $L(f)(z)=N^{-1}\sum_{\sigma_A(w)=z}f(w)$. The bimodule $M_L$ is a copy of $C(\T^d)$ with operations $a\cdot m\cdot b=am\sigma_A^*(b)$ and inner product $\langle m,n\rangle=L(m^*n)$. The map $m\mapsto N^{-1/2}m$ is an isomorphism of $M_L$ onto $X(E)$, and this isomorphism induces isomorphisms of $\TT(M_L)$ onto $\TT(X(E))$ and of the system $(\TT(M_L),\sigma)$ in \cite{LRR} onto our $(\TT(X(E)),\alpha)$. In the presentation of $\TT(M_L)$ used in \cite{LRR}, we need to consider elements $\{u_mv^k:m\in \Z^d,\;k\in \N\}$; such an element $u_mv^k$ lies in $\psi^{\otimes k}(M_L^{\otimes k})$. The isomorphism of $M_L^{\otimes k}$ onto $X(E)^{\otimes k}=X(F_k)$ takes $u_mv^k$ to the function $N^{-k/2}\gamma_m:z\mapsto N^{-k/2}z^m$, and the inner product on $X(F_k)$ is given in terms of $L$ by $\langle y,x\rangle=N^kL^k(\overline y x)$. For $a\in C(\T^d)$, we have
\[
\int a\,d(R^j\varepsilon)=\int\sum_{\sigma_A^j(w)=z}a(w)\,d\varepsilon(z)=\int N^jL^j(a)(z)\,d\varepsilon(z).
\]
Putting this into \eqref{KMSdilation} gives
\begin{align*}
\phi_\varepsilon(u_mv^kv^{*l}u^{*}_n)
&=\delta_{k,l}
e^{-\beta k}\sum_{j=0}^\infty e^{-\beta j}\int N^jL^j\big(N^kL^k\big(\overline{N^{-k/2}\gamma_n}N^{-k/2}\gamma_m\big)\big)\,d\varepsilon\\
&=\delta_{k,l}
\sum_{j=k}^\infty e^{-\beta j}N^{j-k}\int L^j(\gamma_{m-n})\,d\varepsilon.
\end{align*}
The calculation in the third paragraph of the proof of \cite[Proposition~3.1]{LRR} (applied to $A^j$ rather than $A$), shows that with $B:=A^t$ we have
\[
L^j(\gamma_{m-n})=\begin{cases}
0&\text{unless $m-n\in B^j\Z^d$}\\
\gamma_{B^{-j}(m-n)}&\text{if $m-n\in B^j\Z^d$.}\end{cases}
\]
Thus
\begin{align*}
\phi_\varepsilon(u_mv^k&v^{*l}u^{*}_n)
=\delta_{k,l}
\sum_{\{j\geq k\,:\,m-n\in B^j\Z^d\}} e^{-\beta j}N^{j-k}\int \gamma_{B^{-j}(m-n)}\,d\varepsilon\\
&=\delta_{k,l}
\sum_{\{j\geq k\,:\,m-n\in B^j\Z^d\}} e^{-\beta j}N^{j-k}\int z^{B^{-j}(m-n)}(1-Ne^{-\beta})\,d\nu(z).
\end{align*}
Thus $\phi_\varepsilon$ is the state $\psi_{\beta,\nu}$ described in \cite[Proposition~6.1]{LRR}, as claimed.
\end{example}

\section{KMS states at the critical inverse temperature}\label{sec:crit}

\begin{thm}\label{existbetac}
Suppose that $h:Z\to Z$ is a surjective local homeomorphism on a compact Hausdorff space $Z$, $E$ is the topological graph $(Z,Z,\id,h)$, and $X(E)$ is the graph correspondence. Define $\alpha:\R\to \Aut \TT (X(E))$ and $\bar\alpha:\R\to \Aut \OO(X(E))$ in terms of the gauge actions by $\alpha_t=\gamma_{e^{it}}$ and $\bar\alpha_t=\bar\gamma_{e^{it}}$. Take $\beta_c$ as in \eqref{defcrit}. Then there exists a KMS$_{\beta_c}$ state on $(\TT(X(E)),\alpha)$, and at least one such state factors through a KMS$_{\beta_c}$ state of $(\OO(X(E),\bar\alpha)$.
\end{thm}

For the proof we need another variant on \cite[Lemma~10.3]{LR} and \cite[Lemma~2.2]{aHLRS}, where the generating sets $P$ were required to consist of projections.

\begin{lemma}\label{nthsuchlemma}
Suppose $(A,\R,\alpha)$ is a dynamical system, and $J$ is an ideal in $A$ generated by a set $P$ of positive elements which are fixed by $\alpha$. Suppose that there is a family $\mathcal{F}$ of analytic elements such that $\lsp\mathcal{F}$ is dense in $A$, and such that for each $a\in\mathcal{F}$, there is a scalar-valued analytic function $f_a$ satisfying $\alpha_z(a)=f_a(z)a$. If
$\phi$ is a KMS$_{\beta}$ state of $(A,\alpha)$ and $\phi(p)=0$ for all $p\in P$, then $\phi$ factors through a state of $A/J$.
\end{lemma}

\begin{proof}
Consider $p\in P$. Write $p=c^2$ with $c=\sqrt{c}$. Properties of the functional calculus imply that $c\geq 0$ and $\alpha_t(\sqrt{c})=\sqrt{\alpha_t(c)}=\sqrt{c}$. Now we follow the proof of \cite[Lemma~2.2]{aHLRS}: we prove that $\phi(cac)=0$ for all $a\in A$, and then deduce from the KMS condition that $\phi(apb)=\phi(ac^2b)=0$ for all $a,b\in A$. 
\end{proof}

\begin{proof}[Proof of Theorem~\ref{existbetac}]
Choose a decreasing sequence $\{\beta_n\}$ such that $\beta_n\to \beta_c$ and a probability measure $\nu$ on $Z$. Then $K_n:=\int f_{\beta_n}\,d\nu$ belongs to $[1,\infty)$, and $\varepsilon_n:=K_n^{-1}\nu$ satisfies $\int f_{\beta_n}\,d\varepsilon_n=1$. Thus for each $n$, Theorem~\ref{th1} gives us a KMS$_{\beta_n}$ state $\phi_{\varepsilon_n}$ on $(\TT(X(E)),\alpha)$. By passing to a subsequence, we may assume that $\{\phi_{\varepsilon_n}\}$ converges in the weak* topology to a state $\phi$, and \cite[Proposition~5.3.23]{BRII} implies that $\phi$ is a KMS$_{\beta_c}$ state. 

To find a KMS$_{\beta_c}$ state which factors through $\OO(X(E))$, we apply the construction of the previous paragraph to a particular sequence of measures $\varepsilon_n$. Since each $z\mapsto |h^{-n}(z)|$ is continuous \cite[Lemma~2.2]{BRV}, Proposition~2.3 of \cite{FFN} implies\footnote{Strictly speaking, \cite{FFN} require throughout that their space is metric, but their argument for this proposition does not seem to use this.} that there exists $p\in Z$ such that
\begin{equation}\label{choosep}
|h^{-n}(p)|\geq e^{n\beta_c}\quad\text{for all $n\in \N$.}
\end{equation}
Now we let $\delta_p$ be the unit point mass at $p$, and take $\varepsilon_n:=f_{\beta_n}(p)^{-1}\delta_p$. The argument of the first paragraph yields a KMS$_{\beta_c}$ state $\phi$ on $(\TT(X(E)),\alpha)$ which is a weak* limit of the KMS$_{\beta_n}$ states $\phi_{\varepsilon_n}$.

Next we choose a partition of unity $\{\rho_i:1\leq i\leq k\}$ for $Z$ such that $h$ is injective on each $\supp \rho_i$, and take $\xi_i:=\sqrt\rho_i\in X(E)$ as in the proof of Proposition~\ref{tosubinv}. Temporarily, we write $\phi_A$ for the homomorphism of $A=C(Z)$ into $\LL(X(E))$ given by the left action. A calculation like the one in the second paragraph of the proof of Proposition~\ref{tosubinv} shows that for every $a\in A$, $\phi_A(a)$ is the finite-rank operator $\sum_{i=1}^k\Theta_{a\cdot\xi_i,\xi_i}$. Thus the kernel of the quotient map $q:\TT(X(E))\to \OO(X(E))$ is generated by the elements
\begin{align}\label{genkerq}
\pi(a)-(\psi,\pi)^{(1)}\Big(\sum_{i=1}^k\Theta_{a\cdot\xi_i,\xi_i}\Big)
&=\pi(a)-\sum_{i=1}^k\psi(a\cdot\xi_i)\psi(\xi_i)^*\\
&=\pi(a)\Big(1-\sum_{i=1}^k\psi(\xi_i)\psi(\xi_i)^*\Big),
\end{align}
and hence also by the single element $1-\sum_{i=1}^k\psi(\xi_i)\psi(\xi_i)^*$. Equation~\ref{ineqinTT} implies that this single generator is positive in $\TT(X(E))$, so if we can show that $\phi\big(\sum_{i=1}^k\psi(\xi_i)\psi(\xi_i)^*\big)=1$, then it will follow from Lemma~\ref{nthsuchlemma} that $\phi$ factors through $\OO(X(E))$.

We therefore calculate $\phi\big(\sum_{i=1}^k\psi(\xi_i)\psi(\xi_i)^*\big)$. We write $\mu_n$ for the measure $\sum_{j=0}^\infty e^{-\beta_n j}R^j\varepsilon_n$ of Theorem~\ref{th1}\eqref{resolvent}. Then \eqref{equ2} implies that
\begin{align}\label{startcalcphi}
\phi\big(\sum_{i=1}^k\psi(\xi_i)\psi(\xi_i)^*\Big)&=\lim_{n\to\infty}\sum_{i=1}^k\phi_{\varepsilon_n}\big(\psi(\xi_i)\psi(\xi_i)^*\big)\\
&=\lim_{n\to\infty}e^{-\beta_n}\int \sum_{i=1}^k\langle\xi_i,\xi_i\rangle\,d\mu_n.\notag
\end{align}
Since $h$ is injective on each $\supp\xi_i$, we have
\begin{align*}
\sum_{i=1}^k\langle\xi_i,\xi_i\rangle(z)&=\sum_{i=1}^k\sum_{h(w)=z}\overline{\xi_i(w)}\xi_i(w)=\sum_{h(w)=z}\sum_{i=1}^k|\xi_i(w)|^2\\
&=\sum_{h(w)=z}1=|h^{-1}(z)|.
\end{align*}
Thus
\begin{align*}
e^{-\beta_n}\int \sum_{i=1}^k\langle\xi_i,\xi_i\rangle\,d\mu_n
&=e^{-\beta_n}\int |h^{-1}(z)|\,d\mu_n(z)\\
&=\sum_{j=0}^\infty e^{-\beta_n}e^{-\beta_nj}\int |h^{-1}(z)|\,d(R^j\varepsilon_n)(z)\\
&=\sum_{j=0}^\infty e^{-\beta_n(j+1)}\int\sum_{h^{j}(w)=z}|h^{-1}(w)|\,d\varepsilon_n(z).
\end{align*}
Since $\varepsilon_n$ is a point mass, we have
\begin{align*}
e^{-\beta_n}\int \sum_{i=1}^k\langle\xi_i,\xi_i\rangle\,d\mu_n&=\sum_{j=0}^\infty e^{-\beta_n(j+1)}|h^{-(j+1)}(p)|f_{\beta_n}(p)^{-1}\\
&=\sum_{j=1}^\infty e^{-\beta_nj}|h^{-j}(p)|f_{\beta_n}(p)^{-1}.
\end{align*}
Since $f_{\beta_n}(p)=\sum_{j=0}^\infty e^{-\beta_nj}|h^{-j}(p)|$,
we deduce that
\begin{equation}\label{formbetan}
e^{-\beta_n}\int \sum_{i=1}^k\langle\xi_i,\xi_i\rangle\,d\mu_n=\frac{f_{\beta_n}(p)-1}{f_{\beta_n}(p)}.
\end{equation}

We now need to take the limit of \eqref{formbetan} as $n\to \infty$. Since we chose the point $p$ to satisfy \eqref{choosep}, we have
\[
f_{\beta_n}(p)=\sum_{j=0}^\infty e^{-\beta_nj}|h^{-j}(p)|\geq \sum_{j=0}^\infty \big(e^{-(\beta_n-\beta_c)}\big)^j.
\]
Since $e^{-(\beta_n-\beta_c)}\to 1$ as $n\to \infty$, for fixed $J$ we have 
\[
\sum_{j=0}^J\big(e^{-(\beta_n-\beta_c)}\big)^j\to J+1\quad\text{as $n\to \infty$},
\]
and $f_{\beta_n}(p)\to \infty$ as $n\to \infty$. Thus \eqref{formbetan} converges to $1$ as $n\to \infty$, and \eqref{startcalcphi} implies that
\[
\phi\Big(\sum_{i=1}^k\psi(\xi_i)\psi(\xi_i)^*\Big)=1,
\]
as required.
\end{proof}

\begin{remark}\label{compwithT}
Theorem~\ref{existbetac}, and in particular the existence of KMS states on $(\OO(X(E)),\bar\alpha)$ at the inverse temperature $\beta_c$, overlaps with work of Thomsen \cite{Th1}. His results concern KMS states on the $C^*$-algebra of a Deaconu-Renault groupoid, but his Theorem~3.1 identifies his reduced groupoid algebra $C^*_r(\Gamma_h)$ as an Exel crossed product $D\rtimes_{\alpha,L}\N$. In our setting, where the space $Z$ is compact Hausdorff, his $D$ is $C(Z)$, his endomorphism $\alpha$ is given by $\alpha(f)=f\circ h$, and his transfer operator $L$ is given by $L(f)(z)=|h^{-1}(z)|^{-1}\sum_{h(w)=z}f(w)$; Thomsen's Exel crossed product is the Cuntz-Pimsner bimodule of a Hilbert bimodule $M_L$ \cite[Proposition~3.10]{BR}. The bimodule is not quite the same as our $X(E)$, but the map $U:X(E)\to M_L$ given by $(Uf)(z)=|h^{-1}(h(z))|^{1/2}f(z)$ is an isomorphism of $X(E)$ onto $M_L$ (see \cite[\S6]{BRV}). So our $\OO(X(E))$ is naturally isomorphic to the $C^*$-algebra $C_r^*(\Gamma_h)$ in \cite{Th1}. This isomorphism carries the gauge action $\gamma:\T\to \Aut \OO(X(E))$ into the gauge action $\tau$ used in \cite[\S6]{Th1}, and hence our action $\bar\alpha$ is the action $\alpha^F$ of \cite{Th1} for the function $F\equiv 1$ (see the top of \cite[page~414]{Th1}).

For $F\equiv 1$, the sequences $A_F^\phi(k)$ and $B_F^\phi(k)$ in \cite[\S6]{Th1} are given by $A_F^\phi(k)=k=B_F^\phi(k)$, and hence the numbers $A_F^\phi=\lim_{k\to\infty}k^{-1}A_F^\phi(k)$ and $B_F^\phi$ are both $1$. The number $h_m(\phi)$ in \cite[\S6]{Th1} is our $\beta_c$. Thus \cite[Theorem~6.12]{Th1} implies that our system $(\OO(X(E)),\bar\alpha)$ has a KMS$_{\beta_c}$ state. Our approach through $\TT(X(E))$ seems quite different.
\end{remark}

\section{The shift on the path space of a graph}

In this section we consider a finite directed graph $E=(E^0,E^1,r,s)$. In the conventions of \cite{Ra}, we write $E^\infty$ for the set of infinite paths $z=z_1z_2\cdots$ with $s(z_i)=r(z_{i+1})$. The cylinder sets
\[
Z(\mu)=\{z\in E^\infty: z_i=\mu_i\text{ for }i\leq |\mu|\}
\]
form a basis of compact open sets for a compact Hausdorff topology on $E^\infty$. The shift $\sigma:E^\infty\to E^\infty$ is defined by $\sigma(z)=z_2z_3\cdots$. Then $\sigma$ is a local homeomorphism --- indeed, for each edge $e\in E^1$, $\sigma$ is a homeomorphism of $Z(e)$ onto $Z(s(e))$ --- and is a surjection if and only if $E$ has no sinks. Shifts on path spaces were used extensively in the early papers on graph algebras, and in particular in the construction of the groupoid model \cite{KPRR}. Here we shall use them to illustrate our results and those of Thomsen \cite{Th1}.

We consider the topological graph $(E^\infty, E^\infty,\id,\sigma)$, and write $X(E^\infty)$ for the associated Hilbert bimodule over $C(E^\infty)$. The Cuntz-Pimsner algebra $\OO(X(E^\infty))$ is isomorphic to the graph $C^*$-algebra $C^*(E)$ (this is essentially a result from \cite{BRV} --- see the end of the proof below). The relationship between the Toeplitz algebra $\TT(X(E^\infty))$ and the Toeplitz algebra $\TT C^*(E)$ is more complicated.

\begin{prop}\label{includeTalgs}
Suppose that $E$ is a finite directed graph. Then the elements $S_e:=\psi(\chi_{Z(e)})$ and $P_v:=\pi(\chi_{Z(v)})$ of $\TT(X(E^\infty))$ form a Toeplitz-Cuntz-Krieger family. The corresponding homomorphism $\pi_{S,P}$ of $\TT C^*(E)$ into $\TT(X(E^\infty))$ is injective, and $q\circ \pi$ factors through an isomorphism of $C^*(E)$ onto $\OO(X(E^\infty))$. Both isomorphisms intertwine the respective gauge actions of $\T$.
\end{prop}

\begin{proof}
Since the $\chi_{Z(v)}$ are mutually orthogonal projections in $C(E^\infty)$, the $\{P_v:v\in E^0\}$ are mutually orthogonal projections in $\TT(X(E^\infty))$. For $e,f\in E^1$, we have
\[
S_e^*S_f=\psi(\chi_{Z(e)})^*\psi(\chi_{Z(f)})=\pi\big(\langle\chi_{Z(e)},\chi_{Z(f)}\rangle\big).
\]
A calculation shows that $\langle\chi_{Z(e)},\chi_{Z(f)}\rangle$ vanishes unless $e=f$, and then equals $\chi_{Z(s(e))}$; this implies that $S_e^*S_e=P_{s(e)}$, and that the range projections $S_eS_e^*$ and $S_fS_f^*$ are mutually orthogonal. Since the left action satisfies $\chi_{Z(v)}\cdot \chi_{Z(e)}=\chi_{Z(e)}$ when $v=r(e)$, we have $P_vS_eS_e^*=S_eS_e^*$ when $v=r(e)$, and $P_v\geq \sum_{r(e)=v} S_eS_e^*$. Thus $(S,P)$ is a Toeplitz-Cuntz-Krieger family. Since the adjoints $\psi(x)^*$ vanish on the $0$-summand in the Fock module and the representation $\pi$ is faithful there, $P_v\not=\sum_{r(e)=v}S_eS_e^*$ as operators on the Fock module $F(X(E^\infty))$. Thus \cite[Corollary~4.2]{FR} implies that $\pi_{S,P}$ is faithful. Since the gauge actions satisfy $\gamma_z(s_e)=zs_e$ and $\gamma_z(\psi(f))=z\psi(f)$, we have $\pi_{S,P}\circ \gamma=\gamma\circ \pi_{S,P}$.

The left action of $\chi_{Z(\mu)}$ in $X(E^\infty)$ is the finite rank operator $\Theta_{\chi_{Z(\mu)},\chi_{Z(\mu)}}$, and hence we have
\begin{align}\label{compinO}
q\circ\pi(\chi_{Z(\mu)})&=q\circ(\pi, \psi)^{(1)}(\Theta_{\chi_{Z(\mu)},\chi_{Z(\mu)}})\\
&=q\big(\psi^{\otimes|\mu|}(\chi_{Z(\mu)})\psi^{\otimes|\mu|}(\chi_{Z(\mu)})^*\big)\notag\\
&=q(S_\mu S_\mu^*).\notag
\end{align}
Thus every $q\circ\pi(\chi_{Z(\mu)})$ belongs to $C^*(q(S_e),q(P_v))$, and $q\circ \pi(C(E^\infty))$ is contained in $C^*(q(S_e),q(P_v))$. Since $\chi_{Z(v)}=\sum_{r(e)=v}\chi_{Z(e)}$ in $C(E^\infty)$, the calculation \eqref{compinO} shows that $(q\circ S,q\circ P)$ is a Cuntz-Krieger family in $\OO(X(E))$, and the induced homomorphism $\pi_{q\circ S,q\circ P}:C^*(E)\to \OO(X(E^\infty))$ carries the action studied in \cite{aHLRS} to the one we use here. This homomorphism intertwines the gauge actions, and an application of the gauge-invariant uniqueness theorem shows that $\pi_{q\circ S,q\circ P}$ is an isomorphism of $C^*(E)$ onto $\OO(X(E^\infty))$. (The details are in \cite[Theorem~5.1]{BRV}, modulo some scaling factors which come in because the inner product in \cite{BRV} is defined using a transfer operator $L$ which has been normalised so that $L(1)=1$ (see the discussion in \cite[\S9]{BRV}). With our conventions, $L(1)$ would be the function $z\mapsto |\sigma^{-1}(z)|$. Theorem~5.1 of \cite{BRV} extends an earlier theorem of Exel for Cuntz-Krieger algebras \cite[Theorem~6.2]{E2}.)
\end{proof}

\begin{remark}
While Proposition~\ref{includeTalgs} implies that the Toeplitz algebra $\TT(X(E^\infty))$ contains a faithful copy  of $\TT C^*(E)$, Corollary~\ref{restrictionssame} implies that  $\TT (X(E^\infty))$ is substantially larger than $\TT C^*(E)$: for example, there seems to be no way to get $\pi(\chi_{Z(\mu)})$ in $C^*(S,P)$.
\end{remark} 

Since the injections of Proposition~\ref{includeTalgs} intertwine the gauge actions, they also intertwine the dynamics studied in \cite{aHLRS} with those studied here (and there seems little danger in calling them all $\alpha$). Thus applying our results to the local homeomorphism $\sigma$ gives us KMS states on $(\TT C^*(E),\alpha)$ and $(C^*(E),\alpha)$, and we should check that our results are compatible with those of \cite{aHLRS}. 

When $E$ is strongly connected, the system $(C^*(E),\alpha)$ has a unique KMS state, and its inverse temperature is the natural logarithm of the spectral radius $\rho(A)$ of the vertex matrix $A$ of $E$ \cite[Theorem~4.3]{aHLRS} (see also \cite{EFW} and \cite{KW}). So Theorem~\ref{existbetac} implies that, for strongly connected $E$, our critical inverse temperature $\beta_c$ must be $\ln\rho(A)$. Of course, we should be able to see this directly, and in fact it is true for all finite directed graphs. (The restriction to graphs with cycles in the next proposition merely excludes the trivial cases in which $E^\infty$ is empty and $\rho(A)=0$.)

\begin{prop}\label{compbetac}
Suppose that $E$ is a finite directed graph with at least one cycle. Let $A$ denote the vertex matrix of $E$, and let $\sigma$ denote the shift on the infinite-path space $E^\infty$. Then
\[
\frac{1}{N}\ln\Big(\max_{z\in E^\infty}|\sigma^{-N}(z)|\Big)\to \ln \rho(A)\quad\text{as $N\to \infty$.}
\]
\end{prop}

Although this result was easy to conjecture, it does not seem to be so easy to prove, and we need a couple of lemmas. In retrospect, our arguments are similar in flavour to those used in computations of entropy (see \cite[Chapter]{LM}, for example).

For each $z\in E^\infty$, we have $|\sigma^{-N}(z)|=|E^Nr(z)|$. Since $|vE^Nw|=A^N(v,w)$, we have 
\begin{equation}\label{idinvimage}
\max_{z\in E^\infty}|\sigma^{-N}(z)|=\max_{z\in E^0}\sum_{v\in E^0}A^N(v,r(z))=\max_{w\in E^0}\sum_{v\in E^0}A^N(v,w).
\end{equation}
So we are interested in the column sums $\sum_v A^N(v,w)$ of the powers $A^N$. 

When $E$ is strongly connected, $A$ is irreducible in the sense of Perron-Frobenius theory: for each $v,w\in E^0$, there exists $N$ such that $A^N(v,w)>0$. We then have the following consequence of the Perron-Frobenius theorem.

\begin{lemma}\label{colsums}
Suppose that $E$ is a strongly connected finite directed graph with vertex matrix $A$.  Then there is a positive constant $K$ such that, for all $N\in \N$ and $w\in E^0$, 
\begin{equation}\label{colsumsest}
K\Big(\sum_{v\in E^0}A^N(v,w)\Big)\leq \rho(A)^N\leq \max_{u\in E^0}\sum_{v\in E^0}A^N(v,u).
\end{equation}
\end{lemma}

\begin{proof}
Let $x$ be the unimodular Perron-Frobenius eigenvector of $A$, as in \cite[Theorem~1.5]{Se}; then $x$ has positive entries and $Ax=\rho(A)x$. Let $N\in \N$. Then $A^Nx=\rho(A)^Nx$, or equivalently
\begin{equation*}\label{xeigenA}
\sum_{w\in E^0}A^N(v,w)x_w=\rho(A)^N x_v\quad\text{for $v\in E^0$.}
\end{equation*}
Summing over $v$ and interchanging the order of summation gives
\begin{equation}\label{reversesum}
\sum_{w\in E^0}\sum_{v\in E^0}A^N(v,w)x_w=\sum_{v\in E^0}\rho(A)^N x_v=\rho(A)^N.
\end{equation}
Since $x$ is unimodular, \eqref{reversesum} implies that
\[
\rho(A)^N\leq \sum_{w\in E^0} \Big(\max_{u\in E^0}\sum_{v\in E^0}A^N(v,u)\Big)x_w=\max_{u\in E^0}\sum_{v\in E^0}A^N(v,u),
\]
which is the right-hand inequality in \eqref{colsumsest}. For the left-hand inequality, we note that the outside sum on the left-hand side of \eqref{reversesum} is greater than each entry $\sum_{v} A^N(v,w)x_w$, and hence $K:=\min_{w\in E^0}x_w$ has the required property.
\end{proof}

For a more general graph $E$, we define $v\leq w\Longleftrightarrow vE^*w\not=\emptyset$, and then the relation $v\sim w\Longleftrightarrow v\leq w\text{ and }w\leq v$ is an equivalence relation. The equivalence classes or \emph{components} $C$ are then either \emph{trivial} in the sense that $C$ consist of a single vertex $v$ at which there is no loop, or \emph{nontrivial} in the sense that $C:=(C,r^{-1}(C)\cap s^{-1}(C),r,s)$ is a strongly connected graph. 

%The structure theory of Seneta \cite[\S1.2]{Se} allows us to order the vertices so that the vertex matrix $A$ of $E$ is block upper triangular. We choose first the vertices in trivial components which no not feed into nontrivial ones, then the blocks of vertices in nontrivial components which feed only into vertices we have already chosen, then the trivial components which feed only into the vertices, then $\dots$. As in \cite{aHLRS2}, we call such a block decomposition a \emph{Seneta decomposition} of $A$.

If $H$ is a subset of $E^0$, then $H:=(H, r^{-1}(H)\cap s^{-1}(H),r,s)$ is a subgraph of $E$ whose vertex matrix $A_H$ is the restriction of $A$ to $H\times H$. If $H$ is hereditary in the sense that $v\in H$ and $v\leq w$ imply $w\in H$, then the vertex matrix $A$ has a block decomposition
\[
A=\begin{pmatrix}A_{E\backslash H}&A_{E\backslash H,H}\\
0&A_H\end{pmatrix}.
\]

\begin{lemma}\label{RHest}
Suppose that $n\geq 1$ and $E$ is a finite directed graph with $n$  components and at least one cycle. Then there is a constant $\beta_E$ such that
\[
\sum_{v\in E^0}A^N(v,w)\leq \beta_E N^{n-1}\rho(A)^N\quad\text{for all $N\geq 1$ and $w\in E^0$.}
\]
\end{lemma}

\begin{proof}
We prove this by induction on $n$. For $n=1$, the graph is strongly connected, Lemma~\ref{colsums} says there is $K>0$ such that 
\[
\sum_{v\in E^0}A^N(v,w)\leq K^{-1}\rho(A)^N\quad\text{for all $N$ and $w\in E^0$}, 
\]
and then $\beta_E:=K^{-1}$ has the required property. So we suppose that $n\geq 1$, that the lemma is true for every graph with at most $n$ components, and that $E$ has $n+1$ components. We suppose that $C$ is a component which does not receive any edges from other components. Then $C$ is hereditary, and we can decompose $A$
\[
A=\begin{pmatrix}A_{E\backslash C}&A_{E\backslash C,C}\\
0&A_C\end{pmatrix}.
\]
Notice that $A_C$ is either $(0)$ or is irreducible, in which case we can apply the inductive hypothesis to the graph $C=(C,r^{-1}(C),r,s)$ with $n=1$. If $w\in E^0\backslash C$, then every path with source $w$ lies entirely in $E\backslash C$, and for $N\geq 1$
\[
\sum_{v\in E^0}A^N(v,w)=\sum_{v\in E^0\backslash C}A_{E\backslash C}^N(v,w)\leq \beta_{E\backslash C}N^{n-1}\rho(A)^N\leq \beta_{E\backslash C}N^{n}\rho(A)^N.
\]
So it remains to deal with the case where $w\in C$.

Suppose that $w\in C$. Then paths with source $w$ either stay in $C$, or leave $C$ and immediately go to $E^0\backslash C$. Thus
\begin{align*}
\sum_{v\in E^0}&A^N(v,w)=\sum_{v\in C}A_C^N(v,w)+\sum_{k=0}^{N-1}\sum_{v\in E^0}\big(A_{E\backslash C}^{N-1-k}A_{E\backslash C,C}A_C^k\big)(v,w)\\
&=\sum_{v\in C}A_C^N(v,w)+\sum_{k=0}^{N-1}\sum_{v,t\in E^0\backslash C}\sum_{u\in C}A_{E\backslash C}^{N-1-k}(v,t)A_{E\backslash C,C}(t,u)A_C^k(u,w).
\end{align*}
Set $L:=\max_{t,u}A_{E\backslash C,C}(t,u)$. Then
\begin{align*}
\sum_{v\in E^0}&A^N(v,w)\leq\sum_{v\in C}A_C^N(v,w)+\sum_{k=0}^{N-1}\sum_{v,t\in E^0\backslash C}A_{E\backslash C}^{N-1-k}(v,t)L\sum_{u\in C}A_C^k(u,w).
\end{align*}

If $A_C=(0)$, then only the $k=0$ term survives, and $\beta_E:=\beta_{E\backslash C}$ suffices. Otherwise, we can apply the inductive hypothesis to $A_C$ (with $n=1$) and to each $A_{E\backslash C}^{N-1-k}$. With $\beta':=\beta_{E\backslash C}L\beta_C$, this gives:
\begin{align*}
\sum_{v\in E^0}&A^N(v,w)\leq\beta_C\rho(A_C)^N+\sum_{k=0}^{N-1}\sum_{t\in E^0\backslash C}\beta'(N-1-k)^{n-1}\rho(A_{E\backslash C})^{N-1-k}\rho(A_C)^k\\
&\leq \beta_C\rho(A_C)^N+\sum_{k=0}^{N-1}|E^0\backslash C|\beta'(N-1-k)^{n-1}\rho(A_{E\backslash C})^{N-1-k}\rho(A_C)^k.
\end{align*}
Next recall that $\rho(A)=\max\{\rho(A_{E\backslash C}),\rho(A_C)\}$. Thus
\begin{align*}
\sum_{v\in E^0}A^N(v,w)
&\leq\beta_C\rho(A)^N+\sum_{k=0}^{N-1}|E^0\backslash C|\beta'(N-1-k)^{n-1}\rho(A)^{N-1-k}\rho(A)^k\\
&\leq\beta_C\rho(A)^N+N|E^0\backslash C|\beta'(N-1)^{n-1}\rho(A)^{N-1}\\
&\leq \big(\beta_C+|E^0\backslash C|\beta'\big)N^{n}\rho(A)^{N},
\end{align*}
and $\beta_E:=\beta_C+|E^0\backslash C|\beta'$ has the required property. (Since $\beta_E$ is then larger than $\beta_{E\backslash C}$, the cases $w\in E^0\backslash C$ and $A=(0)$ are covered too.)
\end{proof}

\begin{proof}[Proof of Proposition~\ref{compbetac}]
For $z\in E^\infty$, $|\sigma^{-1}(z)|$ is the $r(z)$-column sum of $A$, so we wish to estimate $\sum_v A^N(v,w)$ in terms of $\rho(A)$. We can order the vertices of $E$ to ensure that $A$ is an upper-triangular block matrix whose diagonal blocks are either $(0)$ or $A_C$ for some nontrivial component of $E$ (see \cite[\S2.3]{aHLRS2}, for example). Thus $\rho(A)=\max_C \rho(A_C)$. Because $E$ has at least one cycle, $\rho(A)=\rho(A_C)$ for some nontrivial component $C$. Then from the right-hand estimate in \eqref{colsumsest}, we have for every $N\geq 1$
\begin{equation}\label{LHest}
\rho(A)^N=\rho(A_C)^N\leq\max_{w\in C}\sum_{v\in C}A_C^N(v,w)
\leq\max_{w\in E^0}\sum_{v\in E^0}A^N(v,w).
\end{equation}
Thus Lemma~\ref{RHest} implies that
\[
\rho(A)^N\leq\max_{z\in E^\infty}|\sigma^{-N}(z)|\leq \beta_EN^{n-1}\rho(A)^N\quad\text{for all $N\geq 1$.}
\]
Thus for all $N\geq 1$, we have
\[
\ln\rho(A)=\frac{1}{N}\ln\rho(A)^N\leq\frac{1}{N}\ln\Big(\max_{z\in E^\infty}|\sigma^{-N}(z)|\Big)\leq \frac{1}{N}\ln\big(\beta_EN^{n-1}\rho(A)^N\big).
\]
Since both the left-hand and right-hand sides converge to $\ln\rho(A)$ as $N\to \infty$, the squeeze principle implies that $\frac{1}{N}\ln\big(\max_{z\in E^\infty}|\sigma^{-N}(z)|\big)\to \ln \rho(A)$ also.
\end{proof}

Proposition~\ref{compbetac} implies that, for the shifts $\sigma$ on $E^\infty$, the range $\beta>\beta_c$ in Theorem~\ref{th1} is the same as the range $\beta>\ln \rho(A)$ in \cite[Theorem~3.1]{aHLRS}. When we view $\TT C^*(E)$ as a $C^*$-subalgebra of $\TT(X(E^\infty))$, restricting KMS states of $(\TT(X(E^\infty)),\alpha)$ gives KMS states of $(\TT C^*(E),\alpha)$ with the same inverse temperature. Since we know from \cite[Theorem~3.1]{aHLRS} exactly what the KMS states of $(\TT C^*(E),\alpha)$ are, it is natural to ask which ones arise as the restrictions of states of $(\TT(X(E^\infty)),\alpha)$.

We chose notation in \S\ref{sec:Toeplitz} to emphasise the parallels with \cite[\S3]{aHLRS}, and hence we have a clash when we try to use both descriptions at the same time. So we write $\delta$ for the measure $\varepsilon$ in Theorem~\ref{th1}, and keep $\varepsilon$ for the vectors in $[1,\infty)^{E^0}$ appearing in \cite[Theorem~3.1]{aHLRS}. Otherwise we keep the notation of Theorem~\ref{th1}. 

\begin{prop}\label{restrict}
Suppose that $E$ is a finite directed graph with at least one cycle, and $A$ is the vertex matrix of $E$. Suppose that $\beta>\ln\rho(A)$, and that $\delta$ is a regular Borel measure on $E^\infty$ satisfying $\int f_\beta\,d\delta=1$. Define $\varepsilon=(\varepsilon_v)\in [0,\infty)^{E^0}$ by $\varepsilon_v=\delta(Z(v))$. Take $y=(y_v)\in [1,\infty)^{E^0}$ as in \cite[Theorem~3.1]{aHLRS}. Then $y\cdot \varepsilon=1$, and the restriction of the KMS$_\beta$ state $\phi_\delta$ of Theorem~\ref{th1} to $(\TT C^*(E),\alpha)$ is the state $\phi_\varepsilon$ of \cite[Theorem~3.1]{aHLRS}.
\end{prop}

\begin{proof}
We begin by computing the function $f_\beta\in C(E^\infty)$. For $z\in E^\infty$, we have
\begin{align*}
f_\beta(z)&=\sum_{n=0}^\infty e^{-\beta n}|\sigma^{-n}(z)|
=\sum_{n=0}^\infty e^{-\beta n}|E^nr(z)|
&=\sum_{n=0}^\infty e^{-\beta n}\Big(\sum_{v\in E^0}|E^nv|\chi_{Z(v)}(z)\Big).
\end{align*}
Since $y_v=\sum_{\mu\in E^*v}e^{-\beta |\mu|}$, an application of the monotone convergence theorem shows that
\begin{equation}\label{equateconds}
1=\int f_\beta\,d\delta=\sum_{n=0}^\infty e^{-\beta n}\sum_{v\in E^0}|E^nv|\delta(Z(v))=\sum_{v\in E^0}y_v\varepsilon_v=y\cdot \varepsilon.
\end{equation}

To see that $\phi_\delta$ restricts to $\phi_\varepsilon$, it suffices to compute them both on elements $S_\lambda S_\nu^*$. Since $S_\lambda=\psi^{\otimes|\lambda|}(\chi_{Z(\lambda)})$ belongs to $X(E^\infty)^{\otimes |\lambda|}$, equations \eqref{equ2} and \cite[(3.1)]{aHLRS} imply that $\phi_\delta(S_\lambda S_\nu^*)=0=\phi_\varepsilon(S_\lambda S_\nu^*)$ when $|\lambda|\not=|\nu|$. So we suppose $|\lambda|=|\nu|=n$, say. Then \eqref{equ2} implies that
\[
\phi_\delta(S_\lambda S_\nu^*)=e^{-\beta n}\int \big\langle\chi_{Z(\nu)},\chi_{Z(\lambda)}\big\rangle\,d\mu,
\]
where $\mu=\sum_{k=0}^\infty e^{-\beta k}R^k\delta$. Viewing $X(E^\infty)^{\otimes n}$ as $X(F_n)$, as in Lemma~\ref{idtps}, we can compute
\[
\langle\chi_{Z(\nu)},\chi_{Z(\lambda)}\rangle(z)=\sum_{\sigma^n(w)=z}\overline{\chi_{Z(\nu)}(w)}\chi_{Z(\lambda)}(w)=\delta_{\lambda,\nu}\sum_{\sigma^n(w)=z}\chi_{Z(\lambda)}(w),
\]
and deduce that $\langle\chi_{Z(\nu)},\chi_{Z(\lambda)}\rangle=\delta_{\lambda,\nu}\chi_{Z(s(\lambda))}$. Thus 
\begin{equation}\label{formonbigalg}
\phi_\delta(S_\lambda S_\nu^*)=\delta_{\lambda,\nu}e^{-\beta n}\mu(Z(s(\lambda))).
\end{equation}

So we want to compute $\mu(Z(v))$ for $v\in E^0$. For each $k$, we have
\[
(R^k\delta)(Z(v))=\int \chi_{Z(v)}\,d(R^k\delta)(z)=\int \sum_{\sigma^k(w)=z}\chi_{Z(v)}(w)\,d\delta(z).
\]
We have
\[
\sum_{\sigma^k(w)=z}\chi_{Z(v)}(w)=|vE^kr(z)|=A^k(v,r(z))=\sum_{u\in E^0}A^k(v,u)\chi_{Z(u)}(z).
\]
Thus 
\[
(R^k\delta)(Z(v))=\int \sum_{u\in E^0}A^k(v,u)\chi_{Z(u)}\,d\delta= \sum_{u\in E^0}A^k(v,u)\delta(Z(u)),
\]
and
\begin{align*}
\mu(Z(v))&=\sum_{k=0}^\infty e^{-\beta k}\sum_{u\in E^0} A^k(v,u)\delta(Z(v))\\
&=\sum_{k=0}^\infty e^{-\beta k}(A^n\varepsilon)_v=\big((1-e^{-\beta}A)^{-1}\varepsilon\big)_v.
\end{align*}
Now we go back to \eqref{formonbigalg}, and write down
\begin{equation}\label{formphidelta}
\phi_\delta(S_\lambda S_\nu^*)=\delta_{\lambda,\nu}e^{-\beta n}\big((1-e^{-\beta}A)^{-1}\varepsilon\big)_{s(\lambda)},
\end{equation}
which in the notation of \cite[Theorem~3.1(b)]{aHLRS} is $\delta_{\lambda,\nu}e^{-\beta n}m_{s(\lambda)}$. It follows from this and \cite[(3.1)]{aHLRS} that $\phi_\delta(S_\lambda S_\nu^*)=\phi_\varepsilon(S_\lambda S_\nu^*)$, as required.
\end{proof}

Proposition~\ref{restrict} implies that the system $(\TT(X(E^\infty)),\alpha)$ has many more KMS states than $(\TT C^*(E),\alpha)$:

\begin{cor}\label{restrictionssame}
Suppose that $\beta>\max(\beta_c,\ln\rho(A))$, and that $\delta_1$, $\delta_2$ are regular Borel measures on $E^\infty$ satisfying $\int f_\beta\,d\delta_i=1$. Then $\phi_{\delta_1}|_{\TT C^*(E)}=\phi_{\delta_2}|_{\TT C^*(E)}$ if and only if $\delta_1(Z(v))=\delta_2(Z(v))$ for all $v\in E^0$.
\end{cor}

\begin{proof}
Suppose that $\delta_1$ and $\delta_2$ are as described, and $\phi_{\delta_1}|_{\TT C^*(E)}=\phi_{\delta_2}|_{\TT C^*(E)}$. Then Proposition~\ref{restrict} implies that corresponding $\varepsilon_i$ have $\phi_{\epsilon_1}=\phi_{\epsilon_2}$, and the injectivity of the map $\varepsilon\mapsto\phi_\varepsilon$ from \cite[Theorem~3.1(c)]{aHLRS} says that ${\epsilon_1}={\epsilon_2}$. But this says precisely that $\delta_1$ and $\delta_2$ agree on each $Z(v)$.

On the other hand, if $\delta_1(Z(v))=\delta_2(Z(v))$ for all $v\in E^0$, then the corresponding $\varepsilon_i$ are equal, and the formula \eqref{formphidelta} implies that $\phi_{\delta_1}$ and $\phi_{\delta_2}$ agree on ${\TT C^*(E)}$.
\end{proof}

\begin{cor}\label{allrestrics}
Suppose that $\beta>\max(\beta_c,\ln\rho(A))$. Then every KMS$_\beta$ state of $(\TT C^*(E),\alpha)$ is the restriction of a KMS$_\beta$ state of $(\TT(X(E^\infty)),\alpha)$.
\end{cor}

\begin{proof} Suppose that $\phi$ is a KMS$_\beta$ state on $(\TT C^*(E),\alpha)$. Then \cite[Theorem~3.1]{aHLRS} implies that there is a vector $\varepsilon\in [1,\infty)^{E^0}$ such that $y\cdot \varepsilon=1$ and $\phi=\phi_\varepsilon$. If $\delta$ is a measure on $E^\infty$ such that $\delta(Z(v))=\varepsilon_v$ for all $v\in E^0$ and $\int f_\beta\,d\delta=1$, then Proposition~\ref{restrict} implies that $\phi_\delta|_{\TT C^*(E)}=\phi_\varepsilon$. So it suffices to show that there is such a measure $\delta$.

We can construct measures on $E^\infty$ by viewing it as an inverse limit $\varprojlim(E^n,r_n)$, where $r_n:E^{n+1}\to E^n$ takes $\nu=\nu_1\nu_2\cdots\nu_n\nu_{n+1}$ to $\nu_1\nu_2\cdots\nu_n$. Then any family of measures $\delta_n$ on $E^n$ such that $\delta_{n+1}(Z(\nu)\cap E_{n+1})=\delta_n(Z(\nu))$ for $|\nu|=n$ gives a measure $\delta$ on $E^\infty$ such that $\delta(Z(\nu))=\delta_n(Z(\nu))$ for $|\nu|=n$ (see, for example, Lemma~6.1 of \cite{BLPRR}). We can construct such a sequence by taking $\delta_0=\varepsilon$, inductively choosing weights $w_e$ such that $\sum_{r(e)=v}w_e=\varepsilon_v$, recursively choosing $\{w_{\nu e}\in [0,\infty):\nu e\in E^{n+1}\}$ such that $\sum_{r(e)=s(\nu)}w_{\nu e}=w_\nu$, and setting $\delta_{n+1}(\nu e)=w_{\nu e}$. Now the calculation \eqref{equateconds} shows that $\int f_\beta\,d\delta=y\cdot\varepsilon=1$, and hence $\delta$ has the required properties.
\end{proof}

\section{KMS states below the critical inverse temperature}\label{sec:below}

In Remark~\ref{compwithT}, we showed that our crictical inverse temperature $\beta_c$ is the same as the one found by Thomsen \cite{Th1}. He only considers states of the Cuntz-Pimsner system $(\OO(X(E)),\bar\alpha)$, and we agree that this system has no KMS$_\beta$ states with $\beta>\beta_c$. However, he leaves open the possibility that there are KMS$_\beta$ states with $\beta<\beta_c$. Indeed, he considers also the number
\begin{equation}\label{deflb}
\beta_l:=\limsup_{N\to \infty}\Big(N^{-1}\ln\Big(\min_{z\in Z}|h^{-N}(z)|\Big)\Big),
\end{equation}
and then \cite[Theorem~6.8]{Th1} implies that the KMS states of $(\OO(X(E^\infty)),\bar\alpha)$ all have inverse temperatures in the interval $[\beta_l,\beta_c]$. Since $(\OO(X(E^\infty)),\bar\alpha)=(C^*(E),\alpha)$, we can use examples from \cite{aHLRS2} to see that Thomsen's bounds are best possible.

More precisely, consider the dumbbell graphs
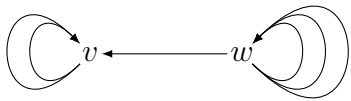
\[
\begin{tikzpicture}
    \node[inner sep=1pt] (v) at (0,0) {$v$};
    \node[inner sep=1pt] (w) at (2,0) {$w$};
    \draw[-latex] (w)--(v);
    \foreach \x in {0,4} {
     \draw[-latex] (v) .. controls +(-1.\x,-1.\x) and +(-1.\x,1.\x) ..(v);}
    \foreach \x in {0,4,8} {
     \draw[-latex] (w) .. controls +(1.\x,-1.\x) and +(1.\x,1.\x) .. (w);}
\end{tikzpicture}
\]
with $m$ loops at vertex $v$ and $n$ loops at vertex $w$. (So in the above picture, we have $m=2$ and $n=3$. This graph was discussed in \cite[Example~6.2]{aHLRS2}, and the one with $m=3$ and $n=2$ in \cite[Example~6.1]{aHLRS2}.) The vertex matrix $A$ of such a graph $E$ is upper triangular and has spectrum $\{m,n\}$. 
For $m\geq n$, the system $(C^*(E),\alpha)$ has a single KMS$_{\ln m}$ state, and this is the only KMS state. 

Now we suppose that $m<n$. Then $\rho(A)=n$, and $(C^*(E),\alpha)$ has two KMS states. The first is denoted by $\psi_{\{w\}}$ in \cite{aHLRS2}, and has inverse temperature $\ln n$. The second factors through the quotient map of $C^*(E)$ onto the $C^*$-algebra of the graph with vertex $v$ and $m$ loops, which is a Cuntz-algebra $\OO_m$. It has inverse temperature $\ln m$. For this graph, we have $\beta_c=\ln\rho(A)=\ln n$. To compute $\beta_l$, we let $z\in E^\infty$. Then
\[
|\sigma^{-N}(z)|=|E^Nr(z)|=
\begin{cases}
m^N&\text{if $r(z)=v$}\\
n^N+\sum_{j=0}^{N-1}n^jm^{N-1-j}&\text{if $r(z)=w$.}
\end{cases}
\]
Since $m<n$, the minimum is attained when $r(z)=v$, and $\min_{z\in E^\infty}|\sigma^{-N}(z)|=m^n$, giving $\beta_l=\ln m$. Thus for this graph, the possible inverse temperatures are precisely the end-points of Thomsen's interval.

\begin{remark}
By adding appropriate strongly connected components between $w$ and $v$ in this last example, we can construct examples for which there are KMS states with  inverse temperatures between $\beta_l$ and $\beta_c$. However, there are number-theoretic constraints on the possible inverse temperatures (see \cite[\S7]{aHLRS2}).
\end{remark}

\end{document}